\documentclass[leqno]{siamart1116}
\usepackage{color}
\usepackage{amssymb}
\usepackage{graphicx}
\usepackage{xspace}
\usepackage{multirow}
\usepackage{bm,bbm}
\usepackage{algpseudocode}
\usepackage[numbers]{natbib}
\let\cite=\citet
\usepackage{subfig}
\usepackage{enumerate}
\usepackage{dsfont}
\usepackage{adjustbox}
\usepackage[normalem]{ulem}
\usepackage{booktabs}

\usepackage{geometry}

\usepackage{mydef}
\newcommand{\FIGS}{.}

\begin{document}
\newcommand\footnotemarkfromtitle[1]{%
\renewcommand{\thefootnote}{\fnsymbol{footnote}}%
\footnotemark[#1]%
\renewcommand{\thefootnote}{\arabic{footnote}}}

\title{Second-order invariant domain preserving approximation of the
  compressible Navier--Stokes equations~\footnotemark[1]}

\author{Jean-Luc~Guermond\footnotemark[2]
  \and Matthias~Maier\footnotemark[2]
  \and Bojan~Popov\footnotemark[2]
  \and Ignacio~Tomas \footnotemark[3]}

\date{Draft version \today}

\maketitle

\renewcommand{\thefootnote}{\fnsymbol{footnote}}

\footnotetext[1]{%
  This material is based upon work supported in part by the National
  Science Foundation grants DMS 1619892, DMS 1620058 and DMS 1912847, by the Air Force
  Office of Scientific Research, USAF, under grant/contract number
  FA9550-18-1-0397, and by the Army Research Office under grant/contract
  number W911NF-15-1-0517, \today}
\footnotetext[2]{%
  Dept. of Mathematics, Texas A\&M University 3368 TAMU, College
  Station, TX 77843, USA}%
\footnotetext[3]{%
  Sandia National Laboratories$^{\S}$, P.O. Box 5800, MS 1320, Albuquerque,
  NM 87185-1320.}
\footnotetext[4]{%
  Sandia National Laboratories is a multimission laboratory managed and
  operated by National Technology \& Engineering Solutions of Sandia, LLC,
  a wholly owned subsidiary of Honeywell International Inc., for the U.S.
  Department of Energy's National Nuclear Security Administration under
  contract DE-NA0003525. This document describes objective technical
  results and analysis. Any subjective views or opinions that might be
  expressed in the paper do not necessarily represent the views of the U.S.
  Department of Energy or the United States Government.}

\renewcommand{\thefootnote}{\arabic{footnote}}

\begin{abstract}
  We present a fully discrete approximation technique for the compressible
  Navier-Stokes equations that is second-order accurate in time and space,
  semi-implicit, and guaranteed to be invariant domain preserving. The
  restriction on the time step is the standard hyperbolic CFL condition, ie
  $\dt \lesssim \calO(h)/V$ where $V$ is some reference velocity scale and $h$ the typical meshsize.
\end{abstract}



\pagestyle{myheadings} \thispagestyle{plain}
\markboth{}{Invariant domain approximation of the compressible Navier--Stokes equations}


\section{Introduction}
The objective of this paper is to present a fully-discrete
approximation technique for the compressible Navier-Stokes equations
that is implicit-explicit, second-order accurate in time and space,
and guaranteed to be invariant domain preserving. The restriction on
the time-step size is the standard hyperbolic CFL condition, \ie
$\dt \lesssim \calO(h)/V$, where $V$ is some reference velocity scale
and $h$ is the typical meshsize. To the best of our knowledge, this
method is the first one that is guaranteed to be invariant domain
preserving under the standard hyperbolic CFL condition and be
second-order accurate in time and space.

Of course there are countless papers in the literature describing techniques
to approximate the time-dependent compressible Navier-Stokes equations, but
there are very few papers establishing invariant domain properties. Among the
latest results in this direction we refer the reader to
\cite{Grapsas_Herbin_Kheriji_Latche_2016} where a first-order method using
upwinding and staggered grid is developed (see Eq.~(3.1) therein). The authors
prove positivity of the density and the internal energy (Lem.~4.4 therein).
Unconditional stability is obtained by solving a nonlinear system involving the
mass conservation equation and the internal energy equation.  One important
aspect of this method is that it is robust in the low Mach regime. A similar
technique is developed in \cite{Gallouet_Gastaldo_Herbin_Latche_2008} for the
compressible barotropic Navier-Stokes equations (see \S3.6 therein).  We also
refer to \cite{Zhang_JCP_2017} where a fully explicit dG scheme is proposed with
positivity on the internal energy enforced by limiting. The invariant domain
properties are proved there under the parabolic time step restriction $\dt
\lesssim \calO(h^2)/\mu$, where $\mu$ is some reference viscosity scale.

The key idea of the present paper is to build on
\citep{Guermond_Nazarov_Popov_Tomas_SISC_2019,
Guermond_Popov_Tomas_CMAME_2019} and use an operator splitting technique to
treat separately the hyperbolic part and the parabolic part of the problem.
The hyperbolic sub-step is treated explicitly and the parabolic sub-step is
treated implicitly. This idea is not new and we refer for instance to
\cite{Demkowicz_etal_1990} for an early attempt in this direction. The novelty
of our approach is that each sub-step is guaranteed to be invariant domain
preserving. In addition, the scheme is conservative and fully-computable (e.g.
the method is fully-discrete and there are no open-ended questions regarding
the solvability of the sub-problems). One key ingredient of our method is that
the parabolic sub-step is reformulated in terms of the velocity and the
internal energy in a way that makes the method conservative, invariant domain
preserving, and second-order accurate
(see \S\ref{Sec:parabolic_implicit_update}).

The remainder of the paper is organized as follows. We recall the compressible
Navier-Stokes model and introduce the notation in~\S\ref{Sec:NS_model}. The
overall principle of the method is summarized in
\S\ref{Sec:Principles_of_the_method}. As usual, the devil is in the details: we
discuss technical aspects of the hyperbolic substep and the parabolic substep
in \S\ref{Sec:hyperbolic_sub_step} and \S\ref{Sec:parabolic_implicit_update},
respectively. The key results of the two sections are
Theorem~\ref{Thm:hyperbolic step} and Theorem~\ref{Thm:parabolic_energy}. We
discuss the full method in \S\ref{Sec:complete_method}. The main statement
summarizing the results of the paper is Theorem~\ref{Thm:main}. The method is
illustrated numerically in \S\ref{Sec:Numerical_illustration}. Some conclusions
and open problems are reported in \S\ref{Sec:conclusions}.


\section{The compressible Navier-Stokes equation} \label{Sec:NS_model}

In this section we define the notation and recall the Navier-Stokes equations.

\subsection{Notation}
The fluid occupies a bounded, polyhedral domain $\Dom$ in $\Real^d$.
The space dimension $d$ is either $2$ or $3$ for simplicity. The dependent
variable is $\bu\eqq (\rho, \bbm, E)\tr \in \Real^{d+2}$, where
$\rho$ is the density, $\bbm$ the momentum, $E$ the total mechanical energy. In
this paper $\bu$ is considered to be a column vector. The velocity is given by
$\bv\eqq \rho^{-1} \bbm$. The quantity $e(\bu) \eqq \rho^{-1}E -
\frac12\|\bv\|_{\ell^2}^2$ is the specific internal energy.

Given some Lipschitz flux
$\polf:\Real^{d+2} \to \Real^{(d+2)\CROSS d}$, $\polf(\bu(\bx))$ is a
matrix with entries $\polf_{ij}(\bu(\bx))$, $1\le i\le d+2$,
$1\le j\le d$ and $\DIV\polf(\bu(\bx))$ is a column vector with
entries
$(\DIV\polf(\bu))_i= \sum_{1\le j\le d}\partial_{x_j}
\polf_{ij}(\bu(\bx))$.
For any $\bn=(n_1\ldots,n_d)\tr\in \Real^d$, we denote by
$\polf(\bu)\bn$ the column vector with entries
$\sum_{1\le l\le d} \polf_{il}(\bu) n_l$, where
$i\in\intset{1}{d+2}$. Given two integers $m\le n$, the symbol
$\intset{m}{n}$ represents the set of integers $\{m,m+1,
\ldots,n\}$.
Given two second-order tensors $\pols$ and $\pole$ in $\Real^{d\CROSS d}$,
we denote the full tensor contraction operation by $\,\pols{:}\pole \eqq
\sum_{i,j\in\intset{1}{d}} \pols_{ij}\pole_{ij}$. As usual $\ba\SCAL\bb
\eqq \sum_{i\in\intset{1}{d}} a_i b_i$ denotes the Euclidean inner-product
in $\Real^d$, and $\ba\otimes \bb$ is the second-order tensor with entries
$(a_i b_j)_{i,j\in\intset{1}{d}}$. For any smooth vector field $\ba:\Dom
\mapsto \Real^d$, $\GRAD\ba$ is the second-order tensor with entries
$(\partial_j a_i)_{i,j\in\intset{1}{d}}$. The Euclidean norm in $\Real^d$
and the Frobenius norm in $\Real^{d\CROSS d}$ are denoted by
$\|\SCAL\|_{\ell^2}$.

\subsection{Model description}
Given some initial time $t_0$ with initial data $\bu_0\eqq
(\rho_0,\bbm_0,E_0)$, we look for $\bu(t) := (\rho,\bbm,E)(t)$ solving the
compressible Navier-Stokes system in some weak sense:
\begin{subequations} \label{NS}
  \begin{align}
    &\partial_t \rho + \DIV(\bv \rho) =  0,
    \label{mass}
    \\
    &\partial_t \bbm + \DIV\big(\bv\otimes \bbm + p(\bu) \polI -
    \pols(\bv)\big) = \bef, \label{momentum}
    \\
    &\partial_t E  + \DIV\big(\bv(E + p(\bu)) - \pols(\bv)\bv +
    \Hflux(\bu)\big) = \bef\SCAL\bv,
    \label{total_energy}
  \end{align}%
where $p(\bu)$ is the pressure, $\polI \in \mathbb{R}^{d\times d}$ is the
identity matrix, $\bef$ is a prescribed external force, $\pols(\bv)$
is the viscous stress tensor and $\Hflux(\bu)$ is the heat-flux. We assume
that the fluid is Newtonian and that the heat-flux follows Fourier's law,
that is to say:
\begin{align*}
  \pols(\bv) &\eqq 2 \mu \pole(\bv) + (\lambda-\tfrac23 \mu) \DIV\bv \polI,
  \qquad
  \pole(\bv) \eqq \GRADs \bv \eqq \tfrac12\big(\GRAD \bv + (\GRAD
  \bv)\tr\big),
  \\
  \Hflux(\bu) &\eqq-c_v^{-1}\kappa \GRAD e.
\end{align*}%
\end{subequations}
The constants $\mu>0$ and $\lambda\ge 0$ are the shear and the bulk
viscosities, respectively. The constant $\kappa$ is the thermal
conductivity and $c_v$ is the heat capacity at constant volume. We
will assume throughout that the coefficient $c_v^{-1}\kappa$ is constant
and does not depend on the state $\bu(t)$.

For the sake of completeness we recall the following standard result regarding
the viscous stress tensor $\pols(\bv)$.
\begin{lemma}
Let $k\eqq \max(0,\frac{d}{3}(1-\frac{3\lambda}{2\mu})) \,\in\, [0,1)$.
Then the following holds true for all smooth vector fields $\bv$ in
$\Real^d$:
\begin{equation}
  \pols(\bv){:}\GRAD \bv \ge 2\mu(1-k)\|\pole(\bv)\|_{\ell^2}^2.
  \label{dissipation_s_e}
\end{equation}
\end{lemma}

\begin{proof}
  We have $\pols(\bv){:}\GRAD \bv =2\mu \GRADs\bv{:}\GRADs\bv + (\lambda
  -\frac23 \mu)(\DIV\bv)^2$ and
  \begin{align*}
  \GRADs\bv{:}\GRADs\bv = \textstyle{\sum_{i,j\in\intset{1}{d}}}
    |\pole(\bv)_{ij}|^2 \ge \textstyle{\sum_{i\in\intset{1}{d}}} |\pole(\bv)_{ii}|^2
    = \sum_{i\in\intset{1}{d}} |\partial_i\bv_{i}|^2\ge \tfrac{1}{d}
    (\DIV\bv)^2.
  \end{align*}
  The result follows readily.
\end{proof}

We assume that the pressure $p(\bu)$ is derived from a complete equation of
state. That is to say, introducing the specific volume $v := \rho^{-1}$,
there exists a specific entropy $\sigma(v,e)$ where $\sigma:\Real^+ \CROSS
\Real^+ \to \Real$ is concave. We assume that the differential of
$\sigma(v,e)$ is consistent with the Gibbs identity $T \diff\sigma = \diff
e + p \diff v$; therefore, setting $s(\rho,e):=\sigma(v,e)$,  we have
$T^{-1} \eqq \tfrac{\partial s}{\partial e}$, $p := -\rho^2 T
\tfrac{\partial s}{\partial \rho}$, see \cite{Meni1989,
Harten_Lax_Levermore_Morokoff_1998} for more details.

The \emph{admissible set} of~\eqref{NS} is
\begin{equation}
  \calA :=\big\{\bu = (\rho,\bbm, E) \in \mathbb{R}^{d+2} \st \rho>0, \
  e(\bu)>0 \big\}.
\end{equation}
This is to say, we expect any reasonable solution $\bu(t)$ of \eqref{NS} to
stay in $\calA$. Following the terminology of \cite{Chueh_Conley_Smoller}
we say that $\calA$ is an invariant domain of \eqref{NS}. Important
properties we want to maintain at the discrete level are thus the
positivity of the density $\rho\ge 0$ and the positivity of the specific
internal energy $e(\bu) = \rho^{-1}E - \frac12 \|\bv\|_{\ell^2}^2$.

We remark in passing that in our formulation the pressure $p(\bu)$ is
necessarily given by an equation of state. This is a crucial property
requried for our splitting technique and analysis. Thus, it is not possible
to directly modify the pressure $p(\bu)$ to, e.\,g., drive the system
\eqref{NS} by an external pressure gradient. Instead, it is necessary to
identify a corresponding external force $\bef$ acting on the momentum as
stated in \eqref{momentum}. For example, a prescribed pressure gradient
$\delta p$ in $\be_1$ direction for a tube of length $L$ then takes the
form $\bef=-\delta p / L \be_1$.

We conclude the section by briefly commenting on boundary conditions for
system \eqref{NS}. For the sake of simplicity and to avoid analytical
technicalities we assume for our analysis that no-slip and thermally
insulating boundary conditions are enforced on the entire boundary
$\front$:
\begin{equation}
  \label{BOUNDARY_CONDITION}
  \bv_{|\front} = \bzero,
  \qquad\qquad
  \Hflux(\bu)\SCAL\bn_{|\front}=0.
\end{equation}
We point out that \eqref{BOUNDARY_CONDITION} indeed adequately closes the
system \eqref{NS}, i.\,e., no further boundary condition has to be
enforced. We refer the reader to \citep[\S
3.5]{Guermond_Nazarov_Popov_Tomas_SISC_2019}, as well as
\S\ref{Sec:hyperbolic_sub_step} and \S\ref{Sec:parabolic_implicit_update}.
In principle it is possible to enforce numerous different boundary
conditions. A careful analysis of all of them is beyond the scope of the
present paper. In our numerical illustrations


\section{Strang splitting and stability properties of the hyperbolic and
parabolic limits}

We will separate the parabolic part and the hyperbolic part of the compressible
Navier-Stokes system \eqref{NS} by using Strang's splitting. To this end, we
first identify a hyperbolic (\S\ref{subse:hyperbolic}) and a parabolic
(\S\ref{subse:parabolic}) limit, then define the corresponding continuous
solution operators $S_1$ and $S_2$, and finally identify associated stability
properties. Both operators are then combined to form a solution operator for
\eqref{NS}; see \S\ref{Sec:Principles_of_the_method}. We make no claim of
originality about the operator splitting technique. The idea is not new and has
been applied in the context of the compressible Navier--Stokes equation by
\cite{Demkowicz_etal_1990} among others. The novel contribution of the present
work is the following:
\begin{enumerate}
  \item[\textup{(i)}]
    The construction of discrete solution operators $S_{1,h}$ and $S_{2,h}$
    that when sequentially compounded yields conservation, preservation of
    the invariant domain properties of the continuous operators (stated
    Assumptions~\ref{Assumption:StabHyp} and \ref{Assumption:StabParab} in
    \S\ref{subse:hyperbolic} and \S\ref{subse:parabolic}), and satisfaction of a
    discrete energy balance.
  \item[\textup{(ii)}]
    Specific choice of transformation of variables at the intermediate step
    making the analysis and an efficient implementation possible.
\end{enumerate}
%


\subsection{Hyperbolic limit}
\label{subse:hyperbolic}

The first asymptotic limit of \eqref{NS} that we discuss is the vanishing
viscosity limit, \ie $\mu,\lambda\to0$, with vanishing external forces
$\bef$. In this case the governing equations for $\bu(t)$ reduce to
\begin{subequations}\label{hyperbolic}
  \begin{align}
    &\partial_t \rho + \DIV(\bv \rho) =  0,
    \label{hyperbolic_mass}
    \\
    &\partial_t \bbm + \DIV(\bv\otimes \bbm + p(\bu) \polI) = \bzero,
    \label{hyperbolic_momentum}
    \\
    &\partial_t E  + \DIV(\bv(E + p(\bu)) = 0,
    \label{hyperbolic_total_energy}
    \\
    &\bv\SCAL \bn_{|\front}=0.
    \label{hyperbolic_bc}
  \end{align}%
\end{subequations}
Here, in the vanishing viscosity limit, the no-slip boundary condition
\eqref{BOUNDARY_CONDITION} is replaced by the slip condition
\eqref{hyperbolic_bc}.
We assume in the following that there exists some Banach space $\calB_1$
with sufficient smoothness so that, provided $\bu_0\in \calB_1\cap \calA$,
some reasonable notion of entropy/viscosity solution of \eqref{hyperbolic}
can be established for some time interval $(t_0,t^*)$. Giving a precise
definition of the functional-space $\calB_1$ is beyond the scope of this
manuscript and somewhat irrelevant for our purpose. The reader is referred
to \cite{LionsBook2,Feireisl2004} for further insights on this very
difficult question. Here, by slight abuse of notation $\calB_1\cap \calA$
shall mean $\{\bv\in \calB_1 \st \bv(\bx)\in \calA\ \text{for \ae} \ \bx
\in \Dom\}$. Let $S_1(\cdot,t_0)$ denote the solution map to
\eqref{hyperbolic}; that is, $S_1(t,t_0)(\bu_0)=\bu(t)$ for \ae $t\in
(t_0,t^*)$. We introduce a stability notion for the solution map
$S_1(\cdot,t_0)$:

\begin{assumption}[Stable hyperbolic solution operator]
\label{Assumption:StabHyp}
Let $\bu_0\in \calB_1\cap \calA$. Recalling that $s$ denotes the specific
entropy, we set $s_{\min}\eqq\essinf_{\bx\in\Dom}s(\rho_0(\bx),e(\bu_0(\bx)))$
and introduce the set:

\begin{equation}
\calC(\bu_0)=\big\{\bu =(\rho,\bbm, E) \st \rho>0, \ e>0,\ s(e,\rho)\ge
s_{\min}\big\}. \label{def_of_calB}
\end{equation}

We make the following assumptions:
\begin{enumerate}
  \item[\textup{(i)}]
    The set $\calC(\bu_0)$ is invariant under $S_1(.,t_0)$ for all
    $\bu_0\in \calA \cap \calB_1$, \ie we have $S_1(t,t_0)(\bu_0)(\bx)\in
    \calC(\bu_0)$ for \ae $\bx\in \Dom$ and \ae $t\in (t_0,t^*)$. We say
    $\calC(\bu_0)$ is an invariant domain of \eqref{hyperbolic}.
  \item[\textup{(ii)}]
    There exists a family of entropy pairs $(\eta,\Eflux)$ (for instance a
    subset of generalized entropies, \cf
    \cite{Harten_Lax_Levermore_Morokoff_1998}) such that the following
    inequality holds in the distribution sense in $\Dom\CROSS(t_0,t^*)$:
    \vspace{-1em}
    \begin{align*}
      \partial_t \eta(S_1(t,t_0)(\bu_0))
      +\DIV(\Eflux(S_1(t,t_0)(\bu_0)))\le 0.
    \end{align*}
  \end{enumerate}
\end{assumption}

\subsection{Parabolic limit}
\label{subse:parabolic}

The second asymptotic regime of interest in this manuscrupt is the diffusive or
parabolic regime. The limit is formally obtained by assuming dominant diffusive
terms and dominant external forces in \eqref{NS}. Then, the governing
equations for $\bu(\bx,t)$ reduce to
\begin{subequations}
  \label{parabolic}
  \begin{align}
    &\partial_t \rho  =  0, \label{parabolic_mass}
    \\
    &\partial_t \bbm - \DIV(\pols(\bv)) = \bef,
    \label{parabolic_momentum}
    \\
    &\partial_t E   + \DIV(\Hflux(\bu)- \pols(\bv) \bv) = \bef\SCAL\bv,
    \label{parabolic_total_energy}
    \\
    &\bv_{|\front}=\bzero, \qquad \Hflux(\bu)\SCAL\bn_{|\front}=0 .
  \end{align}
\end{subequations}%
Since \eqref{parabolic_mass} implies $\rho(\bx,t) = \rho_0(\bx)$ for all
$\bx \in \Dom$, \eqref{parabolic_momentum} is equivalent to $\rho
\partial_t\bv - \DIV(\pols(\bv)) = \bef$. Taking the dot product of
\eqref{parabolic_momentum} and $\bv$ and subtracting the result from
\eqref{parabolic_total_energy} gives $\partial_t (E - \frac12 \rho\bv^2) +
\DIV\Hflux(\bu) - \pols(\bv){:}\GRAD \bv =0$. Consequently,
\eqref{parabolic} is equivalent to solving
\begin{subequations} \label{parabolic_simplified}
  \begin{align}
    &\rho_0\partial_t\bv - \DIV(\pols(\bv)) = \bef,\qquad \bv_{|\front}=\bzero,
    \label{parabolic_momentum_simplified}
    \\
    &\rho_0\partial_t e  - c_v^{-1}\kappa \LAP e  = \pols(\bv){:}\pole(\bv),
    \qquad\partial_n e=0, \label{parabolic_internal_energy_simplified}
    \\
    &E\eqq \rho_0 e + \tfrac12 \rho_0\bv^2.
  \end{align}
\end{subequations}
Notice that
$\partial_t\int_\Dom E \diff x = \int_\Dom \bef\SCAL\bv \diff x$; \ie the
variation of the total energy is equal to the power of the external
sources. Existence and uniqueness of \eqref{parabolic_simplified} can be
established via standard parabolic solution theory,
\cite{gilbarg2015elliptic}. For the sake of argument we will simply assume
that there exists two Banach spaces $\calB_2$ and $\calB_3$ such that the
above problem is well-posed for all $\bu_0\in\calB_2$ and all
$\bef\in\calB_3$. Similarly to the hyperbolic case, we introduce the
solution map $S_2(t,t_0)(\bu_0,\bef)=\bu(t)$ to \eqref{parabolic}. Although
the following assumption could easily be formulated rigorously in
form of a theorem by specifying $\calB_2$ and $\calB_3$, we prefer to make
it an assumption to stay general and avoid distracting technicalities.

\begin{assumption}[Stable parabolic solution operator]
\label{Assumption:StabParab}
Let $\bu_0\in \calA\cap\calB_2$ and $\bef\in \calB_3$. We define
$e_{\min}=\essinf_{\bx\in \Dom}e(\bu_0(\bx))$ and set
\begin{equation}
\calD(\bu_0) :=\big\{\bu =(\rho,\bbm, E) \st \rho>0, \ e \ge
e_{\min}\big\}.
\label{def_of_calC}
\end{equation}
By possibly making $t^*$ smaller we assume that:
\begin{itemize}
  \item[\textup{(i)}]
    The set $\calD(\bu_0)$ is invariant under $S_2(.,t_0)$ for all
    $\bu_0\in \calA\cap\calB_2$ and all $\bef\in\calB_3$, \ie
    $S_2(t,t_0)(\bu_0,\bef)(\bx)\in \calD(\bu_0)$ for \ae $\bx\in \Dom$ and
    \ae $t\in (t_0,t^*)$. We say $\calD(\bu_0)$ is an invariant domain
    for~\eqref{parabolic}.
  \item[\textup{(ii)}]
    The functional setting defining $S_2(t,t_0)$ is smooth enough such that
    \begin{align}
      \label{total_energy_conservation}
      \int_\Dom E(t) \diff x = \int_\Dom E(t_0) \diff x + \int_{t_0}^t
      \int_\Dom \bef\SCAL\bv \diff x.
    \end{align}
  \end{itemize}
\end{assumption}

Our goal in the remainder of the paper is to construct a space and time
approximation that is formally second-order accurate and complies in some
reasonable sense with the stability properties stated in
Assumption~\ref{Assumption:StabHyp} and in
Assumption~\ref{Assumption:StabParab}.

\begin{remark}[Vacuum]
  In this paper we assume that no vacuum forms. It has been established in
  \cite[Thm.~2]{Hoff_Serre_1991} that the compressible Navier-Stokes
  equation may lose continuous dependency with respect to the initial data
  when vacuum occurs. It is shown therein that one can construct initial
  data in one dimension such that continuous dependency is actually lost.
\end{remark}

\begin{remark}[$L^p$ estimates]
  Using $\rho>0$ and the entropy $\eta(\bu)=\rho$ in
  Assumption~\ref{Assumption:StabHyp} we infer the estimate
  $\|\rho\|_{L^\infty(t_0,t^*;L^1(\Dom))}
  \le \|\rho_0\|_{L^\infty(t_0,t^*;L^1(\Dom))}$. Using $\rho>0$, $e>0$,
  \eqref{total_energy_conservation} implies
  $\|\rho e\|_{L^\infty(t_0,t^*;L^1(\Dom))} + \frac12 \|\rho \bv^2\|_{L^\infty(t_0,t^*;L^1(\Dom))}
  = \|\rho_0 e_0\|_{L^1(\Dom)} + \frac12 \|\rho_0 \bv_0^2\|_{L^1(\Dom)}
  +\int_{t_0}^t \int_\Dom \bef\SCAL\bv \diff x.$
\end{remark}

\subsection{Stability of Strang splitting}
\label{Sec:Principles_of_the_method}

We propose to approximate \eqref{NS} in time by using Strang's
operator splitting. To be able to do that without going too much into
the functional analysis details, we add one more assumption which can
always be shown to hold true if $\bu_0$ is smooth enough and $t^*$ is
small enough.

\begin{assumption}[Smoothness compatibility]
  \label{Assumption:smoothness_compatibility}%
  The following holds true for
  \ae $t\in(t_0,t^*)$:
  \begin{itemize}
    \item[\textup{(i)}]
      For all $\bu_0\in\calB_1\cap\calA$, $S_1(t,t_0)(\bu_0)\in \calB_2$.
    \item[\textup{(ii)}]
      For all $\bu_0\in\calB_2\cap\calA$ and all $\bef\in \calB_3$,
      $S_2(t,t_0)(\bu_0,\bef)\in \calB_1$.
  \end{itemize}
\end{assumption}

Let $\dt \in (0,t^*-t_0]$ be some time step and let $\bu_0\in
\calB_1\cap\calA$ be some admissible initial data at time $t_0$. The
version of Strang's splitting technique we consider in this paper consists
of approximating the solution to~\eqref{NS} at $t\eqq t_0+\dt$ as follows:
\begin{equation}
  \label{continuous_Strang}
  S_1(t_0+\dt,t_0+\tfrac12\dt)
  \circ S_2(t_0+\dt,t_0)
  \circ (S_1(t_0+\tfrac12\dt,t_0)(\bu_0),\bef).
\end{equation}
The above operations are well-posed by virtue of
Assumption~\ref{Assumption:smoothness_compatibility}. The following result
is elementary but is essential since it is the template for the
approximation technique that we propose.
\begin{lemma}
  The following holds true for all $\bu_0\in \calB_1\cap\calA$, all
  $\bef\in \calB_3$, all $\dt \in (0,t^*-t_0]$, and \ae $\bx\in\Dom$:
  \vspace{-1em}
  \begin{align*}
    S_1(t_0+\dt,t_0+\tfrac12\dt)\circ S_2(t_0+\dt,t_0)\circ
    (S_1(t_0+\tfrac12\dt,t_0)(\bu_0),\bef)(\bx) \;\in\; \calA.
  \end{align*}
\end{lemma}
\begin{proof}
  By Assumption~\ref{Assumption:StabHyp}\textup{(i)} and of
  Assumption~\ref{Assumption:smoothness_compatibility}\textup{(i)} we have
  $S_1(t_0+\tfrac12\dt,t_0)(\bu_0)\in\calB_2\cap\calC(\bu_0) \subset
  \calB_2\cap\calA$. Similarly, by
  Assumption~\ref{Assumption:StabParab}\textup{(i)} and
  Assumption~\ref{Assumption:smoothness_compatibility}\textup{(ii)} it
  follows that $S_2(t_0+\dt,t_0)\circ
  (S_1(t_0+\tfrac12\dt,t_0)(\bu_0),\bef)\in \calB_1\cap \calD(\bu_0) \subset
  \calB_1\cap \calA$. Finally, the result follows by repeating the first
  argument.
\end{proof}

We now discuss the space and time approximation of the evolution operators
$S_1$ and $S_2$. The two key difficulties to overcome are to ensure that
$\calC(\bu_0)$ remains invariant under the fully discrete version of $S_1$,
and $\calD(\bu_0)$ remains invariant under the fully discrete version of
$S_2$. We describe the discretization of the hyperbolic
step~\eqref{hyperbolic} in \S\ref{Sec:hyperbolic_sub_step}, then we
describe the discretization of the parabolic step~\eqref{parabolic} in
\S\ref{Sec:parabolic_implicit_update}.


\section{Explicit hyperbolic step}
\label{Sec:hyperbolic_sub_step}

In this section we describe the discrete setting that is used to
approximate~\eqref{hyperbolic}. The reader who is familiar with the theory developed in
\cite{Guermond_Nazarov_Popov_Tomas_SISC_2019,Guermond_Popov_Tomas_CMAME_2019}
is invited to skip this section and move on to
\S\ref{Sec:parabolic_implicit_update}.

\subsection{Discrete setting for the space approximation}
\label{Sec:Fin te_element_setting}

For the explicit hyperbolic step we use the exact same setting as described
in~\citep{Guermond_Nazarov_Popov_Tomas_SISC_2019,Guermond_Popov_Tomas_CMAME_2019}.
The method is discretization agnostic and can be implemented with finite
volumes, discontinuous finite elements, and continuous finite elements. To avoid
technicalities when approximating the parabolic problem, we are going to
restrict the presentation to continuous finite elements.
We assume to have at hand a sequence of shape-regular meshes $\famTh$, where
$\calH$ is the index set of the sequence. One may think of $h$ as being the
typical mesh-size. Given some mesh $\calT_h$, we denote by $P(\calT_h)$ a
scalar-valued finite element space with basis functions
$\{\varphi_i\}_{i\in\calV}$. We assume that $P(\calT_h)\subset
C^0(\overline\Dom;\Real)$. We restrict ourselves to continuous Lagrange finite elements
for the sake of simplicity and we assume that $\varphi_i\ge 0$ for all
$i\in \calV$. We denote by $\calV\upbnd$ the set
of the degrees of freedom that are located on the boundary $\front$. The
set $\calV\upint$ is composed of all the interior degrees of freedom. We
introduce the vector-valued approximation space $\bP(\calT_h)\eqq
(P(\calT_h))^{d+2}$. We set
\begin{equation*}
  m_{ij} =\int_\Dom \varphi_i \varphi_j \diff x,\quad
  \bc_{ij} =\int_\Dom \varphi_i \GRAD \varphi_j \diff x,\quad
  \bn_{ij}\eqq\frac{\bc_{ij}}{\|\bc_{ij}\|_{\ell^2}}, \quad
  m_{i} =\int_\Dom \varphi_i\diff x.
\end{equation*}
The definitions of the coefficients
$m_{ij}$, $\bc_{ij}$ and $m_{i}$ for the case of finite volumes and
discontinuous finite element discretizations can be found in
\citep[\S4]{Guermond_Popov_Tomas_CMAME_2019}.

\subsection{Hyperbolic update}
\label{Sec:hyperbolic_update}

Let $t_n$ be some time and $\bu^n\eqq\bu(t_n)$. We now explain how we
approximate the update $S_1(t_{n+1},t_n)(\bu^n)$. First, let $\bu_h^n\eqq
\sum_{i\in \calV}\bsfU_i^n\varphi_i\in \bP(\calT_h)$ be a corresponding
finite element approximation of $\bu^n$. We assume that $\bu_h^n$ is an
admissible state, \ie
\begin{equation*}
  \bsfU_i^n\in \calA, \qquad \forall i\in\calV.
\end{equation*}
Let $\dt$ be the current time step size and set
$t_{n+1}\eqq t_n +\dt$.  Note that $\dt$ has to be chosen for each
time step $t_n$ subject to a suitable hyperbolic CFL condition; see
\eqref{def_dt0}--\eqref{def_of_CFL} and Theorem~\ref{Thm:hyperbolic
  step}. We now construct an approximation
$\bu_h^{n+1} \eqq \sum_{i\in \calV}\bsfU\upnp \varphi_i\in
\bP(\calT_h)$
for the new time step $t_{n+1}$ by combining a low-order approximation
and a high-order approximation through a convex limiting technique
described in
\citep{Guermond_Nazarov_Popov_Tomas_SISC_2019,Guermond_Popov_Tomas_CMAME_2019}.

The low order update is obtained as follows:
\begin{equation*}
  \bsfU_i\upLnp \eqq \bsfU_i^{n}
  + \frac{\dt}{m_i}\sum_{j\in\calI(i)} -\polf(\bsfU_i^{n})\bc_{ij}
  + \frac{\dt}{m_i}\sum_{j\in\calI(i){\setminus}\{i\}} d_{ij}\upLn
  (\bsfU_j^{n} - \bsfU_i^{n}),
\end{equation*}
where $d_{ij}\upLn$ is defined by
\begin{equation}
  \label{Def_of_dij_I}
  d_{ij}^{L,n} := \max\big(\widehat\lambda_{\max}(\bn_{ij},\bsfU_i^n,\bsfU_j^n)
  \|\bc_{ij}\|_{\ell^2},
  \widehat\lambda_{\max}(\bn_{ji},\bsfU_j^n,\bsfU_i^n)
\|\bc_{ji}\|_{\ell^2}\big).
\end{equation}
Here, $\widehat\lambda_{\max}(\bn,\bsfU_L,\bsfU_R)$ is any upper bound on
the maximum wave speed in the Riemann problem with left data $\bsfU_i^n$,
right data $\bsfU_j^n$, and flux $\polf(\bv)\bn_{ij}$. One can use for
instance the two rarefaction approximation discussed in
\cite[Lem.~4.3]{Guermond_Popov_Fast_Riemann_2016} (see also
\cite[Eq.~(4.46)]{Toro_2009}) or any other guaranteed upper bound. For all
$j\in\calI(i){\setminus}\{i\}$ we introduce the auxiliary states
\begin{align}
  \label{def_barstates}%
  \overline{\bsfU}_{ij}^{n} :=
  \frac{1}{2}(\bsfU^{n}_i + \bsfU^{n}_j)
  -(\polf(\bsfU_j^n) - \polf(\bsfU_i^n))\frac{\bc_{ij}}{2 d_{ij}\upLn}.
\end{align}
The following statement is a key result on which the convex limiting
strategy is based.

\begin{lemma}[Invariance of the auxiliary states]\label{Lem:InvarBar}%
  Let $\calU\subset \calA$ be any convex invariant domain for
  \eqref{hyperbolic} such that
  $\bsfU_i^n,\bsfU_j^n\in \calU$. Then the state
  $\overline{\bsfU}_{ij}^{n}$ defined in \eqref{def_barstates} with
  $d_{ij}\upLn$ as defined in \eqref{Def_of_dij_I} belongs to $\calU$.
\end{lemma}

A possibly invariant-domain-violating and formally high-order solution,
$\bu_h\upHnp$, is obtained by appropriately reducing the graph viscosity and
replacing the lumped mass matrix by the full mass matrix (see,
\eg~\citep[\S3.3-\S3.4]{Guermond_Nazarov_Popov_Tomas_SISC_2019} and
\citep[\S6]{Guermond_Popov_Tomas_CMAME_2019}). The final high-order
invariant-domain-preserving update $\bu_h^{n+1}$ is obtained by applying convex
limiting  between the low-order solution
$\bsfU_i\upLnp$ and the high-order solution $\bsfU_i\upHnp$ with relaxed bounds. The
local bounds are computed using the auxiliary states~\eqref{def_barstates}
(see \eg \citep[\S4]{Guermond_Nazarov_Popov_Tomas_SISC_2019} and
\citep[\S7]{Guermond_Popov_Tomas_CMAME_2019}). In the numerical illustrations reported
at the end of the paper we limit the density from above and from below and the
specific entropy from below. The relaxation technique for the bounds is
explained in \citep[\S4.7]{Guermond_Nazarov_Popov_Tomas_SISC_2019} and
\citep[\S7.6]{Guermond_Popov_Tomas_CMAME_2019}. For further reference we
introduce
\begin{equation} \label{def_dt0}
  \dt_0(\bu_h^n)\eqq \min_{i\in \calV}\frac{m_i}{2 |d_{ii}\upLn|},
  \qquad \text{with}\qquad
  d_{ii}\upLn \eqq-\sum_{j\in\calI(i){\setminus}\{i\}} d_{ij}\upLn.
\end{equation}
The ratio $\dt/\dt_0(\bu_h^n)$ is henceforth denoted $\text{CFL}$ and called Courant-Friedrichs-Lewy number:
\begin{equation}
  \text{CFL}\eqq \frac{\dt}{\dt_0(\bu_h^n)}. \label{def_of_CFL}
\end{equation}
Let $S_{1h}(t_n+\dt,t_n): \bP(\calT_h) \to \bP(\calT_h)$ denote the
nonlinear operator defined by setting $S_{1h}(t_n+\dt,t_n)(\bu_h^n)\eqq
\bu_h^{n+1}$. The key result regarding the hyperbolic update is the
following.

\begin{theorem}[Invariance]
  \label{Thm:hyperbolic step}
  Let $\bu_h^n\in\calA$ and let $\calC(\bu_h^n)$ be as defined in
  \eqref{def_of_calB}.
  \begin{itemize}
    \item[\textup{(i)}]
      If no relaxation is applied on the entropy bounds, then
      $S_{1h}(t_n+\dt,t_n)(\bu_h^n)\in \calC(\bu_h^n)$ for all $\dt \le
      \dt_0(\bu_h^n)$. In other words, $\calC(\bu_h^n)$ is invariant under
      $S_{1h}(t_n+\dt,t_n)$ if $\textup{CFL}\le 1$.
    \item[\textup{(ii)}]
      In case of relaxation of the entropy bounds in the convex limiter,
      there exists $c(h)$ with $\lim_{h\to 0} c(h) =1$ and $s_{\min} \ge c(h)s_{\min}$ so that
      the same statement holds with the constraint $s(\rho,e)\ge s_{\min}$
      in \eqref{def_of_calB} replaced by $s(\rho,e)\ge c(h)s_{\min}$.
    \item[\textup{(iii)}]
      In both cases $\calA$ is invariant under $S_{1h}(t_n+\dt,t_n)$
      provided that $\dt \le \dt_0(\bu_h^n)$.
    \end{itemize}
\end{theorem}

\begin{remark}[Second-order in time] \label{Rem:SSP}
  In practice the method is made second-order accurate in time by using a
  strong stability preserving explicit Runge Kutta method. For instance it
  is sufficient to use SSPRK(2,2) (\ie Heun's scheme) to achieve
  second-order accuracy in time. This is done as follows: one computes
  $\bw_h^1=S_{1h}(t_n+\dt,t_n)(\bu_h^n)$ and
  $\bw_h^2=S_{1h}(t_n+2\dt,t_n+\dt)(\bw_h^1)$ and one sets $\bu_h^{n+1} =
  \frac12 \bu_h^n + \frac12 \bw_h^2.$
\end{remark}


\section{Implicit parabolic step} \label{Sec:parabolic_implicit_update}

We now describe the discrete setting that is used to approximate the
parabolic step~\eqref{parabolic}. We use the same finite element setting
that was introduced in \S\ref{Sec:Fin te_element_setting}.

\subsection{Density and velocity update}

Let again $\bu_h^n\eqq \sum_{i\in \calV}\bsfU_i^n\varphi_i\in \bP(\calT_h)$
be a finite element approximation of $\bu^n$. We assume that $\bu_h^n$ is
an admissible state, \ie
\begin{equation}
  \bsfU_i^n\in \calA, \qquad \forall i\in\calV.
\end{equation}
Let $\dt$ be the chosen hyperbolic time step size (see
\S\ref{Sec:hyperbolic_sub_step}) for $t_n$. We now construct an
approximation $\bu_h^{n+1} = \sum_{i\in\calV} \bsfU^{n+1}_i\varphi_i$ of
$S_2(t_n+\dt,t_n)(\bu^n, \bef)$ as follows. Since the evolution equation for
the density in~\eqref{parabolic} is $\partial_t\rho=0$, the density is
updated by setting
\begin{equation}
  \varrho_i^{n+1} \eqq \varrho_i^n,\qquad \forall i\in\calV.
  \label{mass_parabolic_discrete}
\end{equation}
Next, the velocity $\bv^n$ has to be updated. For this, we introduce the
bilinear form associated with viscous dissipation,
\begin{align*}
  a(\bv,\bw) \eqq \int_\Dom\pols(\bv){:}\pole(\bw) \diff x, \qquad
  \bv,\bw\in \bH^1_0(\Dom)\eqq H_0^1(\Dom;\Real^{d}).
\end{align*}
Let $\{\be_k\}_{k\in\intset{1}{d}}$ be the canonical Cartesian basis of
$\Real^d$.  For any $i\in\calV$ and $j\in\calI(i)$ we define the
$d\CROSS d$ matrix $\polB_{ij}\in \Real^{d\CROSS d}$ by setting
\begin{equation}
  (\polB_{ij})_{kl} \eqq a(\varphi_j \be_l,\varphi_i \be_k) \eqq  \int_\Dom
  \pols(\varphi_j \be_l){:} \GRADs(\varphi_i \be_k) \diff x,\qquad \forall
  k,l\in\intset{1}{d}.
\end{equation}
Let $\bef_h^{n+\frac12}\eqq\sum_{j\in\calV}\bsfF_j^{n+\frac12}\varphi_j \in
\bP(\calT_h)$ be an approximation of $\bef(t_n+\frac12 \dt)$ (at least
second-order accurate in time and space). We use the Crank-Nicolson
technique to compute $\bu_h^{n+1}$. More precisely we solve for the unknown
$\bsfV^{n+\frac{1}{2}}$ given by the following linear system:
\begin{subequations}
  \label{parabolic_discrete}
  \begin{equation}
    \label{mt_parabolic_discrete}
    \begin{cases}
      \varrho^{n}_i m_i \bsfV^{n+\frac{1}{2}} +
      \tfrac12 \dt\sum_{j\in\calI(i)} \polB_{ij} \bsfV^{n+\frac{1}{2}} =
      m_i \bsfM_i^{n} + \tfrac12 \dt m_i \bsfF_i^{n+\frac12},
      & \forall i\in \calV\upint
      \\[0.3em]
      \bsfV_i^{n+\frac{1}{2}} = \bzero, &  \forall i\in \calV\upbnd,
    \end{cases}
  \end{equation}
  where $\bsfU^n_i\qqe (\varrho^n_i, \bM^n_i, E^n_i)$, and set
  \begin{equation}
    \label{vel_parabolic_discrete}
    \bsfV_i^{n+1} \eqq 2 \bsfV^{n+\frac{1}{2}} - \bsfV_i^{n},\qquad
    \bsfM_i^{n+1} \eqq \varrho^{n+1}_i\bsfV_i^{n+1},\qquad \forall i\in\calV.
  \end{equation}
\end{subequations}
We then introduce $\bv_h^{n+\frac{1}{2}} \eqq \sum_{i\in \calV}
\bsfV_i^{n+\frac{1}{2}}\varphi_i$ and define
\begin{equation}
  \label{def_Ki_nplusone}
  \sfK_i^{n+\frac{1}{2}} \eqq \frac{1}{m_i}\int_\Dom
  \pols(\bv^{n+\frac{1}{2}}){:}\pole(\bv^{n+\frac{1}{
  2}}) \varphi_i \diff x,
  \qquad \forall i\in\calV.
\end{equation}
Notice that $\sum_{i\in\calV} m_i \sfK_i^{n+\frac{1}{2}} =
a(\bv^{n+\frac{1}{2}},\bv^{n+\frac{1}{2}})$ owing to the partition
of unity property. The main properties of the above definitions are summarized
in the following result.
\begin{lemma}[Velocity update]
  \textup{(i)} For every $i\in\calV$ we have $\sfK_i^{n+\frac{1}{2}} \ge 0$.
  \textup{(ii)}
  The following global energy balance holds true:
  \begin{equation}
    \label{kinetic_energy_balance}
    \sum_{i\in\calV}\tfrac12 m_i\varrho^{n}_i (\bsfV_i^{n+1})^2
    + \dt a(\bv^{n+\frac{1}{2}},\bv^{n+\frac{1}{2}}) =
    \sum_{i\in\calV}\tfrac12 m_i\varrho^{n}_i(\bsfV_i^{n})^2
    + \sum_{i\in\calV} \dt
    m_i\bsfF_i^{n+\frac12}\SCAL\bsfV_i^{n+\frac{1}{2}}.
  \end{equation}
\end{lemma}

\begin{proof}
  \textup{(i)} The inequality $\sfK_i^{n+\frac{1}{2}} \ge 0$ is a
  consequence of~\eqref{dissipation_s_e} and $\varphi_i\ge 0$.
  \textup{(ii)} We take the dot product of \eqref{mt_parabolic_discrete}
  with $2 \bsfV_i^{n+\frac{1}{2}}$ and recalling that $\bsfV^{n+\frac{1}{2}} = \frac12(\bsfV_i^{n+1} + \bsfV_i^{n})$
we obtain for every $i\in\calV\upint$
  \begin{align*}
    \tfrac12 m_i\varrho^{n}_i (\bsfV_i^{n+1})^2
    + \dt a(\bv^{n+\frac{1}{2}}, \bsfV_i^{n+\frac{1}{2}}\varphi_i)
    = \tfrac12 m_i\varrho^{n}_i(\bsfV_i^{n})^2
    + \dt m_i\bsfF_i^{n+\frac12}\SCAL \bsfV_i^{n+\frac{1}{2}}.
  \end{align*}
  For every $i\in\calV\upbnd$ we have $\bsfV_i^{n+\frac{1}{2}} =\bzero$,
  which in turn implies that $\bsfV_i^{n+1} = -\bsfV_i^n$, \ie
  $(\bsfV_i^{n+1})^2 = (\bsfV_i^n)^2$. Moreover, we have
  $a(\bv^{n+\frac{1}{2}}, \bsfV_i^{n+\frac{1}{2}}\varphi_i)=0$ and
  $\bsfF_i^{n+\frac12}\SCAL\bsfV_i^{n+\frac{1}{2}} =0$. Hence, for every
  $i\in\calV\upbnd$ we have
  \begin{align*}
    \tfrac12 m_i\varrho^{n}_i (\bsfV_i^{n+1})^2
    + \dt a(\bv^{n+\frac{1}{2}},\bsfV_i^{n+\frac{1}{2}}\varphi_i)
    = \tfrac12 m_i\varrho^{n}_i(\bsfV_i^{n})^2
    + \dt m_i\bsfF_i^{n+\frac12}\SCAL \bsfV_i^{n+\frac{1}{2}}.
  \end{align*}
  Summing over $i\in\calV$ and using the partition of unity property
  ($\sum_{i\in\calV} \varphi_i=1$) yields~\eqref{kinetic_energy_balance}.
\end{proof}

\begin{remark}[Approximation order]
  The update $\bsfV_i^{n+1}$ constructed by \eqref{parabolic_discrete} is
  formally second-order accurate in time and space since
  \eqref{mt_parabolic_discrete} is a Crank-Nicolson time step.
\end{remark}


\subsection{Internal energy update (first-order)}

The update of the internal energy entails some subtleties regarding the
minimum principle when using the second-order Crank-Nicolson time stepping.
Therefore, we first formulate the method with the backward Euler time
stepping. The second-order extension is presented in~\S\ref{Sec:energy_CN}.
Let us introduce the bilinear form associated with the thermal diffusion
\begin{align*}
  b(e,w) \eqq c_v^{-1}\kappa \int_\Dom \GRAD e\SCAL \GRAD w \diff x,\qquad
  \forall e,w\in H^1(\Dom).
\end{align*}
For any $i\in\calV$ and $j\in\calI(i)$ we set
\begin{equation}
  \beta_{ij}\eqq b(\varphi_j,\varphi_i).
\end{equation}
Notice that the partition of unity property implies that $\beta_{ii} =
-\sum_{j\in\calI(i){\setminus}\{i\}}\beta_{ij}$. This implies in particular
that for all $v_h\eqq \sum_{j\in \calV} \sfV_j \varphi_j \in P(\calT_h)$ we
have
\begin{equation}
  b(v_h,\varphi_i) = \sum_{j\in\calI(i){\setminus}\{i\}} \beta_{ij}
  (\sfV_j-\sfV_i).
\end{equation}
This expression will be useful to prove the minimum principle on the
internal energy. We further assume that
\begin{equation}
  \beta_{ij}\le 0, \quad \forall i\ne j \in \calV.
\end{equation}
This condition is known to be satisfied for meshes composed of simplices in
two and three space dimensions under the so-called acute
angle condition, \cf \eg \cite[\S5.2]{BrKoK:08},
\cite[Eq.~(2.5)]{XuZik:99}. This is in particular true for Delaunay meshes.
Although it can be done, it is not the purpose of this paper to relax this
condition.

Recalling the viscous dissipation $\sfK_i^{n+\frac{1}{2}}$ defined in
\eqref{def_Ki_nplusone}, we now construct a low-order update of the
internal energy $\sfe_i\upLnp$ as follows. For all $i\in\calV$ first
set $\sfe_i^{n}\eqq (\varrho^{n}_i)^{-1} E_i^n -
\tfrac12\|\bsfV_i^{n}\|_{\ell^2}^2$, then
solve the linear system
\begin{align}
  \label{low_int_energy_2}
  & m_i \varrho_i^{n}(\sfe_i\upLnp - \sfe_i^{n})+ \dt \sum_{j\in\calI(i)}
  \beta_{ij}\sfe_j\upLnp =  \dt m_i\sfK_i^{n+\frac{1}{2}},
  \qquad \forall i\in \calV.
\end{align}
Recall that the boundary conditions \eqref{parabolic_internal_energy_simplified}
together with the partition of unity property imply that
\begin{equation}
  \label{internal_energy_identity}
  \sum_{i\in\calV} m_i \varrho_i^{n}(\sfe_i\upLnp - \sfe_i^{n}) =
  \dt\sum_{i\in\calV} m_i\sfK_i^{n+\frac{1}{2}}
  =\dt a(\bv^{n+\frac{1}{2}},\bv^{n+\frac{1}{2}}).
\end{equation}
This identity is used in the proof of Theorem~\ref{Thm:parabolic_energy}.
\begin{lemma}[Minimum principle]
  \label{lem:parabolic_energy_low}
  Let $\bsfU^{n}$ be an admissible state.  Then for all $\dt>0$:
  \begin{align*}
    \min_{j\in\calV}\sfe_j\upLnp \ge \min_{j\in\calV} (\sfe_j^n
    + \tfrac{\dt}{\varrho_j^{n}} \sfK_j^{n+\frac{1}{2}})
    \ge \min_{j\in\calV} \sfe_j^n \ge 0.
  \end{align*}
\end{lemma}

\begin{proof}
  Recalling that $\sum_{j\in\calI(i)}\beta_{ij}=0$, we infer that
  \begin{align*}
    m_i \varrho_i^{n}(\sfe_i\upLnp - \sfe_i^{n})
    +\dt \sum_{j\in\calI(i){\setminus}\{i\}}
    \beta_{ij}(\sfe_j\upLnp -\sfe_i\upLnp)
    =  \dt m_i\sfK_i^{n+\frac{1}{2}},
  \end{align*}
  Let $i$ be the index in $\calV$ where $\sfe_i\upLnp$ is minimal. Then
  $0\ge \sum_{j\in\calI(i){\setminus}\{i\}} \beta_{ij}(\sfe_j\upLnp
  -\sfe_i\upLnp)$ because we have assumed that $\beta_{ij}\le 0$ for all
  $j\in\calI(i){\setminus}\{i\}$. Moreover, the definition of
  $\sfK_i^{n+\frac{1}{2}}$ implies that $\sfK_i^{n+\frac{1}{2}}\ge 0$
  since we assumed $\varphi_i\ge 0$. All this implies that
  \begin{align*}
    m_i \varrho_i^{n}(\sfe_i\upLnp - \sfe_i^{n}) & \ge m_i
    \varrho_i^{n}(\sfe_i\upLnp - \sfe_i^{n})
    + \dt \sum_{j\in\calI(i){\setminus}\{i\}}
    \beta_{ij}(\sfe_j\upLnp -\sfe_i\upLnp)
    =  \dt m_i\sfK_i^{n+\frac{1}{2}}\ge 0.
  \end{align*}
  In conclusion
  $\min_{j\in\calV}\sfe_j\upLnp \qqe \sfe_i\upLnp \ge
  \sfe_i^{n}+\tfrac{\dt}{\varrho_i^{n}}\sfK_i^{n+\frac{1}{2}} \ge
  \min_{j\in\calV}
  \big(\sfe_j^{n}+\tfrac{\dt}{\varrho_j^{n}}\sfK_j^{n+\frac{1}{2}}
  \big)$.
\end{proof}


\subsection{Internal energy update (Second-order)}
\label{Sec:energy_CN}

We now explain how to approximate the internal energy with a second-order
Crank-Nicolson time stepping scheme. This is done by combining the
low-order update and the second-order update using flux-corrected transport
limiting (FCT); the reader is referred to \eg \cite{Boris_books_JCP_1973,
Zalesak_1979, KuzminLoehnerTurek2004}.

We start by defining the high-order update of the internal energy,
$\sfe_i\upHnp,$ as follows: We first compute
$\sfe_i{\upHnph}$ by solving
\begin{align}
  \label{high_int_energy_CN}
  & m_i \varrho_i^{n}(\sfe_i{\upHnph} - \sfe_i^{n})+\tfrac12\dt
  \sum_{j\in\calI(i)} \beta_{ij}\sfe_i{\upHnph}
  = \tfrac12 \dt m_i\sfK_i^{n+\frac{1}{2}}, \qquad \forall i\in \calV.
\end{align}
and then set
\begin{equation*}
  \sfe_i\upHnp = 2\sfe_i{\upHnph} - \sfe_i^n, \qquad \forall i\in \calV.
\end{equation*}
In general, positivity properties for Crank-Nicolson schemes can only be
guaranteed under highly restrictive time-step size constraints. We do not
assume that such time-step conditions are met. We just assume that the time-step size
is dictated by the CFL constraints of the hyperbolic part. We thus
resort to flux-corrected transport limiting, or alternatively convex
limiting, to preserve positivity properties. Rewriting
\eqref{high_int_energy_CN} by multiplying \eqref{high_int_energy_CN} by 2
and replacing $\sfe_i\upHnph$ by
$\frac12(\sfe_i\upHnp+\sfe_i^n)$ gives:
\begin{align}
  \label{high_int_energy_2}
  & m_i \varrho_i^{n}(\sfe_i\upHnp - \sfe_i^{n})+\tfrac12\dt
  \sum_{j\in\calI(i)} \beta_{ij}(\sfe_j\upHnp  + \sfe_j^{n})
  =  \dt m_i\sfK_i^{n+\frac{1}{2}}, \qquad \forall i\in \calV.
\end{align}
We then take the difference between \eqref{high_int_energy_2} and
\eqref{low_int_energy_2} to obtain
\begin{align*}
  & m_i \varrho_i^{n}(\sfe_i\upHnp - \sfe_i\upLnp) =
  -\tfrac12\dt
  \sum_{j\in\calI(i)} \beta_{ij}(\sfe_j\upHnp + \sfe_j^{n} - 2
  \sfe_j\upLnp).
\end{align*}
Setting $A_{ij}\eqq -\tfrac12\dt \beta_{ij}(\sfe_j\upHnp-\sfe_i\upHnp +
\sfe_j^{n} -\sfe_i^{n} - 2 \sfe_j\upLnp+2 \sfe_i\upLnp)$, the above
identity reads
\begin{align*}
  m_i \varrho_i^{n}(\sfe_i\upHnp - \sfe_i\upLnp) =
  \sum_{j\in\calI(i){\setminus}\{i\}}A_{ij}.
\end{align*}
Introducing $\sfe^{n,\min}\eqq\min_{j\in \calV} \sfe_j\upn$ we then define the
FCT limiter coefficients as follows:
\begin{subequations}
  \begin{align}
    & P_i^-\eqq \sum_{j\in\calI(i){\setminus}\{i\}} \min(A_{ij},0),
    && Q_i^-\eqq m_i\varrho_i^n (\sfe^{n,\min} -\sfe_i\upLnp),
    \label{fct_internal_energy_1}
    \\
    & \limiter_i^+ = 1,
    && \limiter_i^-\eqq \min\big(1,\tfrac{Q_i^-}{P_i^-}\big).
    \label{fct_internal_energy_2}
  \end{align}
\end{subequations}
Note that $P_i^- \leq 0$ and $Q_i^- \leq 0$ (owing to
Lemma~\ref{lem:parabolic_energy_low}), therefore $\limiter_i^- \geq 0$. By
virtue of the definition of $\limiter_i^-$ the inequality $\limiter_i^-
P_i^- \geq Q_i^-$ always holds true:
\begin{align}
  \label{CoreFCTbound}
  \limiter_i^- P_i^- = \min\big(1,\tfrac{Q_i^-}{P_i^-}\big) P_i^- =
  - \min\big(1,\tfrac{Q_i^-}{P_i^-}\big) |P_i^-| =
  - \min(|P_i^-|,-Q_i^-) \geq Q_i^-
\end{align}
The high-order update of the internal energy is now defined by setting
\begin{equation}
  \label{internal_energy_parabolic_discrete}
  m_i \varrho_i^{n}(\sfe_i\upnp - \sfe_i\upLnp) =
  \sum_{j\in\calI(i){\setminus}\{i\}} \limiter_{ij} A_{ij}, \qquad
  \limiter_{ij} \eqq
  \begin{cases}
    \min(\limiter_i^{+},\limiter_j^{-}),& \text{if $A_{ij}\ge 0$},\\
    \min(\limiter_i^{-},\limiter_j^{+}),& \text{if $A_{ij}<0 $}.
  \end{cases}
\end{equation}

\begin{lemma}[Minimum principle]
  \label{Lem:internal_energy_parabolic_discrete}
  The quantity $\sfe^{n+1}$ computed in
  \eqref{internal_energy_parabolic_discrete} satisfies
  \begin{equation}
    \label{Eq:Lem:internal_energy_parabolic_discrete}
    \min_{j\in\calV} \sfe_j^{n+1} \ge \sfe^{n,\min}\eqq\min_{j\in \calV}
    \sfe_j\upn.
  \end{equation}
\end{lemma}

\begin{proof}
  The above definitions imply
  \begin{align*}
    m_i \varrho_i^{n}(\sfe_i\upnp - \sfe_i\upLnp) &\ge
    \sum_{j\in\calI(i){\setminus}\{i\}} \limiter_{ij} \min(A_{ij},0) \ge
    \limiter_{i}^{-}  \sum_{j\in\calI(i){\setminus}\{i\}}\min(A_{ij},0)
    = \limiter_{i}^{-} P_i^- \ge Q_i^-,
  \end{align*}
  where we have used that $\limiter_{ij} \leq \limiter_{i}^{-}$, the
  definition of $P_i^-$, and the inequality \eqref{CoreFCTbound}. This shows
  that the limiting enforces $m_i \varrho_i^{n} \sfe_i^{n+1} \ge m_i
  \varrho_i^{n} \sfe^{n,\min}$, \ie $\sfe_i^{n+1} \ge \sfe^{n,\min}$. This
  in turn implies that $\min_{i\in\calV}\sfe_i^{n+1} \ge \sfe^{n,\min} =
  \min_{j\in\calV}\sfe_j\upn.$
\end{proof}


\subsection{Total energy update}

Once the internal energy is updated according
to~\eqref{internal_energy_parabolic_discrete}, the total energy can be
updated by setting
\begin{equation}
  \label{energy_parabolic_discrete}
  E_i^{n+1} = \varrho_i^{n+1}\sfe_i^{n+1} +
  \tfrac12\varrho^{n}_i \|\bsfV_i^{n+1}\|_{\ell^2}^2,\qquad \forall i\in \calV.
\end{equation}
The main result of \S\ref{Sec:parabolic_implicit_update} is the following.
\begin{theorem}[Positivity and conservation]
  \label{Thm:parabolic_energy}
  Let $\bsfU^{n}$ be an admissible state. Let $\bsfU^{n+1}$ be the stated
  constructed by \eqref{mass_parabolic_discrete} -
  \eqref{vel_parabolic_discrete} - \eqref{energy_parabolic_discrete}, with
  the velocity update defined in \eqref{parabolic_discrete} and the
  internal energy update defined
  in~\eqref{internal_energy_parabolic_discrete}.
  Then, $\bsfU^{n+1}$ is an admissible state, \ie $\bsfU_i^{n+1}\in\calA$
  for all $i\in \calV$ and all $\dt$, and the following holds for all
  $i\in\calV$ and all $\dt$:
  \begin{subequations}
    \label{eq:lem:parabolic_energy}
    \begin{align}
      \varrho_i^{n+1} &= \varrho_i^n > 0, \qquad \forall i\in\calV,
      \label{mass:lem:parabolic_energy}
      \\
      \min_{j\in\calV} \sfe_j^{n+1}
      & \ge \min_{j\in\calV} \sfe_j^{n} > 0,
      \\
      \sum_{i\in\calV} m_i \sfE_i^{n+1}
      & = \sum_{i\in\calV} m_i \sfE_i^{n}
      + \sum_{i\in\calV} \dt m_i\bsfF_i^{n+\frac12}\SCAL
      \bsfV_i^{n+\frac{1}{2}}.
      \label{eq2:lem:parabolic_energy}
    \end{align}
  \end{subequations}
\end{theorem}
\begin{proof}
  \textup{(i)} Since by assumption $\bsfU^{n}_i\in\calA$, we have
  $\varrho_i^n > 0$, whence $\varrho_i^{n+1} > 0$. \\
  \textup{(ii)} We have proved that $\min_{j\in\calV} e_j^{n+1} \ge
  \min_{j\in\calV} e_j^{n}\ge 0$ in Lemma~\ref{lem:parabolic_energy_low}. \\
  \textup{(iii)} We have established in \eqref{kinetic_energy_balance} that
  \begin{align}
    \label{proof_ke_balance}
    \sum_{i\in\calV}\tfrac12 m_i\varrho^{n}_i (\bsfV_i^{n+1})^2
    + \dt a(\bv^{n+\frac{1}{2}},\bv^{n+\frac{1}{2}})
    = \sum_{i\in\calV}\tfrac12 m_i\varrho^{n}_i(\bsfV_i^{n})^2
    + \sum_{i\in\calV} \dt m_i\bsfF_i^{n+\frac12}\SCAL \bsfV_i^{n+\frac{1}{2}}.
  \end{align}
  Recalling that $A_{ij} =-A_{ji}$ and $\limiter_{ij}=\limiter_{ji}$, we
  sum~\eqref{internal_energy_parabolic_discrete} over $i\in\calV$ and
  obtain
  \begin{align*}
    \sum_{i\in\calV} m_i \varrho_i^n\sfe_i^{n+1} = \sum_{i\in\calV} m_i
    \varrho_i^n \sfe_i\upLn.
  \end{align*}
  Invoking the identity~\eqref{internal_energy_identity} shows
  \begin{align}
    \label{proof_se_balance}
    \sum_{i\in\calV} m_i \varrho_i^n \sfe_i^{n+1}
    = \sum_{i\in\calV} m_i \varrho_i^n\sfe_i^{n}
    + \dt a(\bv^{n+\frac{1}{2}},\bv^{n+\frac{1}{2}}).
  \end{align}
  Adding \eqref{proof_ke_balance} and \eqref{proof_se_balance}
  gives~\eqref{eq2:lem:parabolic_energy}.
\end{proof}

We introduce a discrete nonlinear solution operator $S_{2h}(t_n+\dt,t_n):
\bP(\calT_h)\CROSS\bP(\calT_h)  \to \bP(\calT_h)$ by setting
$S_{2h}(t_n+\dt,t_n)(\bu_h^n,\bef_h^{n+\frac12})\eqq \bu_h^{n+1}$.
Theorem~\ref{Thm:parabolic_energy} can then be rephrased as follows.

\begin{corollary}[Invariance]
  \label{Cor:parabolic_invariance}
  Let $\bu_h\in \bP(\calT_h)\cap\calA$ and let $\bef_h^{n+\frac12}\in
  \bP(\calT_h)$. Then $\calD(\bu_h^n)$ is invariant under
  $S_{2h}(t_n+\dt,t_n)$ for all $\dt$, \ie
  $S_{2h}(t_n+\dt,t_n)(\bu_h,\bef_h^{n+\frac12}) \in \calD(\bu_h^n)\subset
  \calA$ for all $\dt>0$.
\end{corollary}

\begin{remark}[Definition of $\sfe^{\min}$]
  The definition of $\sfe^{\min}$ in \eqref{fct_internal_energy_1} can be
  slightly strengthened. The lower
  bound~\eqref{Eq:Lem:internal_energy_parabolic_discrete} holds for any
  number $\sfe^{\min}$ chosen in the interval $[\min_{j\in \calV} \sfe_j^n,
  \min_{j\in \calV} \sfe_j\upLn]$. However, selecting $\sfe^{\min}$ too
  close to $\min_{j\in \calV} \sfe_j\upLn$ degenerates the accuracy order
  of the method to $\calO(\dt)$ in the $L^\infty(\Dom)$-norm. The numerical
  experiments reported in the paper are computed with $\sfe^{\min}\eqq
  \min_{j\in \calV} \sfe_j^n$.
\end{remark}

\begin{remark}[Energy]
Lemma~\ref{Lem:internal_energy_parabolic_discrete} establishes that the
minimum of the internal energy grows monotonically and
Theorem~\ref{Thm:parabolic_energy} states that the temporal variation of
the total energy is equal to the power of the sources. This implies in
essence that a fully discrete counterpart of
\eqref{total_energy_conservation} holds true, which is exactly what one
should expect.
\end{remark}


\section{Complete method}
\label{Sec:complete_method}

We now put all the pieces together and state the main ressult of the paper.
Let $S_{1h}^{(2)}$ be a version of $S_{1h}$ that is at least second-order
accurate in time as discussed in Remark~\ref{Rem:SSP}. Let $\bu_h^n\in
\bP(\calT_h)$ be an admissible state and let $\bef_h^{n+\frac12}\in
\bP(\calT_h)$.  Let us fix some number $\textup{CFL}>0$, which we call Courant-Friedrichs-Lewy number,
 and let $\dt_0(\bu_h^n)$ be defined in \eqref{def_dt0}. The time step
$\dt$ is chosen by setting
\begin{equation}
\dt \eqq \textup{CFL} \CROSS \dt_0(\bu_h^n). \label{def_of_dt_complete_method}
\end{equation}
The update $\bu_h^{n+1}
\in\bP(\calT_h)$ is computed as follows:
\begin{equation}
  \bu_h^{n+1} = S_{1h}^{(2)}(t_n+\dt,t_n+\tfrac12\dt)
  \circ S_{2h}(t_n+\dt,t_n)
  \circ (S_{1h}^{(2)}(t_n+\tfrac12\dt,t_n)(\bu_h^n),\bef_h^{n+\frac12}).
\end{equation}

\begin{theorem}[Invariance]
  \label{Thm:main}
  Let $\bu_h^n\in \bP(\calT_h)\cap\calA$ and $\bef_h^{n+\frac12}\in
  \bP(\calT_h)$. Then $\bu_h^{n+1}\in \calA$ provided $\textup{CFL}$ is
  small enough.%
\end{theorem}

\begin{proof}
  From Theorem~\ref{Thm:hyperbolic step} we infer that
  $S_{1h}^{(2)}(t_n+\tfrac12\dt,t_n)(\bu_h^n)\in\calA$ if $\textup{CFL}$ is small
  enough. For example, for the SSPRK(2,2) and SSPRK(3,3) methods this holds
  with $\textup{CFL}=2$. From Corollary~\ref{Cor:parabolic_invariance} we
  infer that $\bw_h\eqq S_{2h}(t_n+\dt,t_n)\big(
  S_{1h}^{(2)}(t_n+\tfrac12\dt,t_n)(\bu_h^n,\bef_h^{n+\frac12}) \big) \,\in\,
  \calA$ without any further restriction on $\dt$. Using again
  Theorem~\ref{Thm:hyperbolic step} we infer that
  $S_{1h}^{(2)}(t_n+\dt,t_n+\tfrac12\dt)(\bw_h) \in \calA$ provided
  $\frac{\dt}{2} \le \dt_0(\bw_h)$, \ie $\textup{CFL}\le 2
  \dt_0(\bw_h)/\dt_0(\bu_h^n)$.
\end{proof}

\begin{remark}[\textup{CFL}]
Showing that Theorem~\ref{Thm:main} holds with a CFL number that is
uniform with respect to the mesh size, \ie $\dt_0(\bw_h)/\dt_0(\bu_h^n)$
can be bounded uniformly, would necessitate to prove some uniform bounds
on $\bw_h$. Except under very
restrictive smallness assumptions on data, to the best of our knowledge this is  a very challenging open
problem that is well beyond the scope of the present paper.
\end{remark}


\section{Numerical illustration}
\label{Sec:Numerical_illustration}

We illustrate the approximation technique with a number of convergence
tests and a computation of a shocktube benchmark problem.

\subsection{Implementation details}

All the tests reported below are done with the ideal gas equation of state,
$s(\rho,e) = \log(e^{\frac{1}{\gamma-1}}\rho^{-1})$, with $\gamma = 1.4$.
This in turn implies that $p=(\gamma-1)\rho e$, as well as $c_p =
\frac{\gamma}{\gamma-1}$, and $c_v = \frac{1}{\gamma-1}$. We also assume that the ratio
$\frac{\mu c_p}{\kappa}\qqe P_r$, called Prandtl number, is constant.
Hence $c_v^{-1} \kappa = P_r^{-1} \frac{c_p}{c_v}\mu = \frac{\gamma}{P_r} \mu$.
The bulk viscosity $\lambda$ is set to $0$.

All the computations are done with continuous $\polP_1$ elements. The
high-order method uses the entropy viscosity commutator described in
\citep[(3.15)--(3.16)]{Guermond_Nazarov_Popov_Tomas_SISC_2019} with the
entropy $\rho s$.  Upper and lower bounds on the density are enforced by
using the method described in
\citep[\S4.4]{Guermond_Nazarov_Popov_Tomas_SISC_2019}. The relaxation of
the bounds on the density is done by using the technique described in
\citep[\S4.7]{Guermond_Nazarov_Popov_Tomas_SISC_2019}. The minimum
principle on the specific entropy $\exp((\gamma-1)s) \ge
\exp((\gamma-1)s^{\min})$ is enforced by proceeding as in
\citep[\S4.6]{Guermond_Nazarov_Popov_Tomas_SISC_2019} with the constraint
$\Psi(\bsfU):= \rho e - \varrho^{\min} \rho^\gamma\ge 0$. The lower bound
on the specific entropy for all $i\in\calV$ is set with $\varrho^{\min}_i
\eqq \min_{j\in \calI(i)} \rho_i^n e_i^n/(\rho_i^n)^\gamma$ and
further relaxed by using
\citep[Eq.~(4.14)]{Guermond_Nazarov_Popov_Tomas_SISC_2019}. The positivity
of the internal energy is guaranteed by the minimum principle on the
specific entropy, \ie no limiting on the internal energy is done.
High-performance implementations of the hyperbolic solver are
available in form of open source software documented
in \cite{maier2020massively,MaierTomas2020}.

The demonstration code used here has not been parallelized. The linear
system are solved by using the preconditioned CG version of PARDISO
(phase=23). The solution tolerance is set to $10^{-10}$ (parm(4)=102). The
reader is referred to \cite{Cosmin_Schenk_Lubin_SISC_2014}.

\subsection{1D Convergence tests}
\label{Sec:1D_viscous_shock}

We estimate the convergence properties of the method on a smooth solution.
We consider a one-dimensional viscous shockwave problem that has an exact
solution which is described in \cite{Becker_1922}. A partial English
translation of \citep{Becker_1922} and other exact solutions are found in
\cite{Johnson_JFM_2013}. The Navier-Stokes system \eqref{NS} is solved over
the real line with no source term, $\bef=\bzero$.

One key assumption of \citep{Becker_1922} is that the Prandtl number
$P_r\eqq \frac{\mu c_P}{\kappa}$ is fixed and equal to $\frac34$. Recall
that $\mu$ is the shear viscosity and $\kappa$ is the thermal conductivity.
The bulk viscosity $\lambda$ is set to $0$.

We first construct a steady state solution. Let $\rho(x)$ be the density,
$v(x)$ the velocity, and $e(x)$ the internal energy. Let $v_0$ be the
velocity at infinity on the left ($v_0\eqq\lim_{x\to -\infty} v(x)$) and
let $v_1$ be the velocity at infinity on the right ($v_1\eqq\lim_{x\to
+\infty} v(x)$).  We assume that $v_0>v_1$. We define $v_{01}\eqq \sqrt{v_0
v_1}$.  Let $\rho_0$ be the density at infinity on the left. Since the
solution is time-independent, the momentum is constant, say $m_0$. In the
context of the above assumptions, it is shown in
\citep[Eq.~(30.a)]{Becker_1922} (see also
\citep[Eq.~(3.6)]{Johnson_JFM_2013}) that the velocity profile $\Real \ni
x\mapsto v(x)$ is defined implicitly as the solution to the
following equation:
\begin{equation}
  \label{Becker_1D_solution}
  x = \frac{2}{\gamma+1} \frac{\kappa}{m_0 c_v}
  \Big\{\frac{v_0}{v_0-v_1}\log\Big(\frac{v_0-v(x)}{v_0-v_{01}}\Big)
  - \frac{v_1}{v_0-v_1}\log\Big(\frac{v(x)-v_1}{v_{01}-v_1}\Big)\Big\}.
\end{equation}
This equation is solved numerically to high accuracy by using a Newton
technique. Notice that by convention, \eqref{Becker_1D_solution} implies
that $v(0) = v_{01}$. Once $v(x)$ is known, the density and the internal
energy at $x$ are given by
\begin{equation}
  \rho(x) = \frac{m_0}{v(x)},\qquad
  e(x) = \frac{1}{2\gamma}\Big(\frac{\gamma+1}{\gamma-1}v_{01}^2 -
  v^2(x)\Big).
\end{equation}
To obtain a time-dependent solution, which is computationally more
challenging than solving a steady state solution, we construct a moving
wave as follows. We first introduce the constant translation velocity
$v_\infty$ and we define
\begin{equation}
  \bu(x,t)\eqq \begin{pmatrix}\rho(x-v_\infty t)
  \\
  \rho(x-v_\infty t)(v_\infty + v(x-v_\infty t))
  \\
  \rho(x-v_\infty t)(e(x-v_\infty t) + \frac12(v_\infty+v(x-v_\infty t))^2
  \end{pmatrix}.
\end{equation}
The field $\bu$ solves~\eqref{NS} for any $v_\infty$ since the
Navier-Stokes equations are Galilean invariant. This solution is used for
instance in \cite{Dumbser_Comput_Fluids_2010} for verification purposes.

We now compare the above solution to numerical simulations using the
following parameters $\gamma=1.4$, $\mu=0.01$, $v_\infty=0.2$, $v_0=1$,
$\rho_0=1$. This gives $m_0=1$.  Instead of enforcing $v_1$, we choose the
pre-shock Mach number $M_0=3$, which then gives
$v_1=\frac{\gamma-1+2 M_0^{-2}}{\gamma+1}$; see
\citep[Eq.~(2.10)]{Johnson_JFM_2013}.
Notice that $\kappa=\frac{\mu c_p}{P_r}$ with $P_r=\frac34$.  We use the
truncated domain $[-1,1.5]$ (the larger the domain the higher the accuracy
that can be reached on extremely fine grids).
Inhomogeneous Dirichlet boundary conditions are enforced on all conserved
quantities $\bu=(\rho,\bbm,E)$ at the left and right boundary (see
\S\ref{Sec:NS_model}).
The simulations are run until $t=3$. The distance traveled by the shock is
$0.6$. For $q\in \{1,2,\infty\}$, we compute a consolidated error indicator
at the final time by adding the relative error in the $L^q$-norm of the
density, the momentum, and the total energy as follows:
\begin{align}
  \label{def_delta_t}
  \delta_q(t):=\frac{\|\rho_h(t)-\rho(t)\|_{L^q(\Dom)}}{\|\rho(t)\|_{L^q(\Dom)}} +
  \frac{\|\bbm_h(t)-\bbm(t)\|_{\bL^q(\Dom)}}{\|\bbm(t)\|_{\bL^q(\Dom)}} +
  \frac{\|E_h(t)-E(t)\|_{L^q(\Dom)}}{\|E(t)\|_{L^q(\Dom)}}.
\end{align}
We show in Table~\ref{Table:1D_viscous_SW} the results for 7 uniform
grids. The coarsest grid has $50$ grid points and the finest has $3200$
grid points. The number of grid points is denoted by $\Nglob$. We observe
second-order convergence in time and space in all the norms, as expected.

\begin{table}[ht]\small \centering
  \caption{%
    1D Viscous schockwave, $\mathbb{P}_1$ uniform meshes, Convergence tests,
    $t=3$, $\text{CFL}=0.4$.}
  \label{Table:1D_viscous_SW}
  \begin{tabular}{rcccccc}
    \toprule
    \Nglob & $\delta_1(t)$ & rate & $\delta_2(t)$ & rate & $\delta_\infty(t)$ & rate \\[0.5em]
     50    & 5.85E-02      & --   & 3.11E-01      & --   & 8.28E-03           & --   \\
     100   & 2.50E-02      & 1.23 & 1.91E-01      & 0.71 & 2.82E-03           & 1.55 \\
     200   & 4.83E-03      & 2.37 & 3.27E-02      & 2.54 & 5.13E-04           & 2.46 \\
     400   & 1.07E-03      & 2.17 & 9.79E-03      & 1.74 & 9.32E-05           & 2.46 \\
     800   & 2.52E-04      & 2.09 & 2.29E-03      & 2.10 & 2.02E-05           & 2.21 \\
    1600   & 6.20E-05      & 2.02 & 5.76E-04      & 1.99 & 4.89E-06           & 2.05 \\
    3200   & 1.55E-05      & 2.00 & 1.46E-04      & 1.98 & 1.23E-06           & 1.99 \\
    \bottomrule
  \end{tabular}\\[1em]

  \caption{%
    2D Viscous schockwave, $\mathbb{P}_1$ nonuniform Delaunay meshes, $t=3$,
    $\text{CFL}\in\{0.4,0.9\}$.}
  \label{Table:2D_viscous_SW}
  \begin{tabular}{crcccccc}
    \toprule
    CFL
    & \Nglob & $\delta_1(t)$ & rate & $\delta_2(t)$ & rate & $\delta_\infty(t)$ & rate \\[0.5em]
    \multirow{5}{*}{0.4} 
    & 4458   & 8.99E-03      & --   & 1.49E-02      & --   & 1.20E-01           & --   \\
    & 17589  & 1.35E-03      & 2.76 & 3.04E-03      & 2.31 & 3.23E-02           & 1.91 \\
    & 34886  & 5.19E-04      & 2.80 & 1.47E-03      & 2.13 & 1.44E-02           & 2.36 \\
    & 69781  & 2.45E-04      & 2.17 & 7.20E-04      & 2.05 & 7.93E-03           & 1.72 \\
    & 139127 & 1.04E-04      & 2.47 & 3.71E-04      & 1.93 & 3.27E-03           & 2.56 \\[0.5em]
    \multirow{5}{*}{0.9} 
    & 4458   & 6.99E-03      & --   & 2.03E-02      & --   & 1.58E-01           & --   \\
    & 17589  & 9.51E-04      & 2.91 & 3.39E-03      & 2.61 & 3.61E-02           & 2.15 \\
    & 34886  & 3.98E-04      & 2.54 & 1.60E-03      & 2.20 & 1.55E-02           & 2.47 \\
    & 69781  & 1.79E-04      & 2.30 & 7.54E-04      & 2.17 & 8.23E-03           & 1.83 \\
    & 139127 & 8.17E-05      & 2.28 & 3.67E-04      & 2.09 & 3.28E-03           & 2.67 \\
    \bottomrule
  \end{tabular}
\end{table}

\subsection{2D Convergence tests} \label{Sec:2D_viscous_shock}

We use again the exact shockwave solution described in
\S\ref{Sec:1D_viscous_shock} to verify the method in two-space dimensions.
This test is also meant to verify that the method is genuinely second-order
accurate on non-uniform meshes. Here we use nonuniform Delaunay
triangulations. The convergence tests are done in the truncated domain
$\Dom=(-0.5,1)\CROSS(0,1)$.
In addition to inhomogeneous Dirichlet boundary conditions on the left and
right side we enforce periodic boundary conditions on $\{y=0\}$ and
$\{y=1\}$.
The length of the domain in the $x$-direction is slightly smaller than for
the one-dimensional tests reported above. We do not expect to saturate the
relative error indicators $\delta_1$, $\delta_2$ and $\delta_\infty$ due to
boundary effects in this smaller computational domain since we restrict the
meshsize not to be smaller than $1/425$.  We use 5 meshes. These meshes are
not nested to eliminate the risk of observing super-convergence effects.
This makes having consistent convergence rates more difficult and therefore
tests the robustness of the method. The meshsizes for these meshes are
approximately $0.02, 0.01, 0.0707, 0.05, 0.003536$. The results are
reported in Table~\ref{Table:2D_viscous_SW} for the two CFL numbers $0.4$
and $0.9$. We observe that the method is second-order accurate both in time
and space, for both CFL numbers, and in all error norms.

\subsection{2D shocktube test}

As a final numerical test we simulate the interaction of a shock with a
viscous boundary layer. The test case we consider has been introduced in
the literature by \cite{Daru_Tenaud_2000} and is further documented in
\cite{Daru_Tenaud_2009}. It is essentially a shocktube problem. The tube is
the square cavity $\Dom=(0,1)^2$ with a diaphragm at $\{x=\frac12\}$
separating it in two parts. The fluid is initially at rest. The state on
the left-hand side of the diaphragm is $\rho_L=120$, $v_L=0$,
$p_L=\rho_L/\gamma$. The right state is $\rho_R=1.2$, $v_R=0$,
$p_R=\rho_R/\gamma$. We use the ideal gas equation of state
$p=(\gamma-1)\rho e$ with $\gamma=1.4$. The bulk viscosity is set to $0$.
The Prandtl number is $Pr=0.73$.
No-slip and thermally insulating boundary conditions
\eqref{BOUNDARY_CONDITION} are enfourced throughout.
The diaphragm is broken at $t=0$. A shock, a contact and a rarefaction wave
are created. The viscous shock and the contact move to the right.  The
rarefaction wave moves to the left.  As the shock and the contact waves
progress to the right they create thin viscous boundary layers on the top
and the bottom walls of the tube.  The shock hits the right wall at
approximately $t\approx 0.2$ and is then reflected.  The shock interacts
with the contact discontinuity on its way back to the left.  Complex
interactions occur and the contact discontinuity stays stationary close to
the right wall thereafter. The shock wave then continues its motion to the
left and interacts with the viscous boundary layer which it created while
moving to the right. This interaction is very strong and a lambda shock is
formed as a result. We refer to \citep[\S6]{Daru_Tenaud_2000} and
\citep[\S5\&\S6]{Daru_Tenaud_2009} for full descriptions of the various
mechanisms at play in this problem.

The computations reported in this paper are done in the half domain
$(0,1)\CROSS(0,\frac12)$. Symmetry with respect to the horizontal axis
$\{y=\frac12\}$ is obtained by enforcing a slip boundary condition instead
of the no-slip boundary condition \eqref{BOUNDARY_CONDITION}.
This is achieved algebraically by simply replacing the homogeneous
Dirichlet condition $\bsfV_i^{n+\frac{1}{2}}=\bzero$ in
\eqref{parabolic_discrete} by $\bn\SCAL\bsfV_i^{n+\frac{1}{2}} = \bzero$
for the upper boundary at $\{y=\frac12\}$.
The CFL number used for these computations is $0.95$ (see
\eqref{def_of_CFL} and \eqref{def_of_dt_complete_method}). The computations
are done with nonuniform meshes that are progressively refined. The meshes
are highly nonuniform to concentrate the grid points in the right part of
the cavity. In mesh~1 the meshsize is about $0.0007$ on $\{ 0.3\le x\le 1,
y =0\}$ and $0.0014$ on $\{ 0.5\le x\le 1, y =0.5\}$ ($359388$ grid
points). The meshsize in the second mesh is about $0.0005$ on $\{ 0.3\le
x\le 1, y =0\}$ and $0.001$ on $\{ 0.5\le x\le 1, y =0.5\}$ ($684996$ grid
points).  For mesh~3 the meshsize is about $0.0004$ on $\{ 0.3\le x\le 1, y
=0\}$ and $0.001$ on $\{ 0.5\le x\le 1, y =0.5\}$ ($859765$ grid points).

\begin{figure}[hbt]\setlength{\fboxsep}{0pt}%
  \setlength{\fboxrule}{0.5pt}%
  \begin{center}
    \subfloat[Mesh 1, $t=0.6$.]%
      {\fbox{\includegraphics[width=0.32\textwidth,trim = 450 0 0 590,clip]{\FIGS/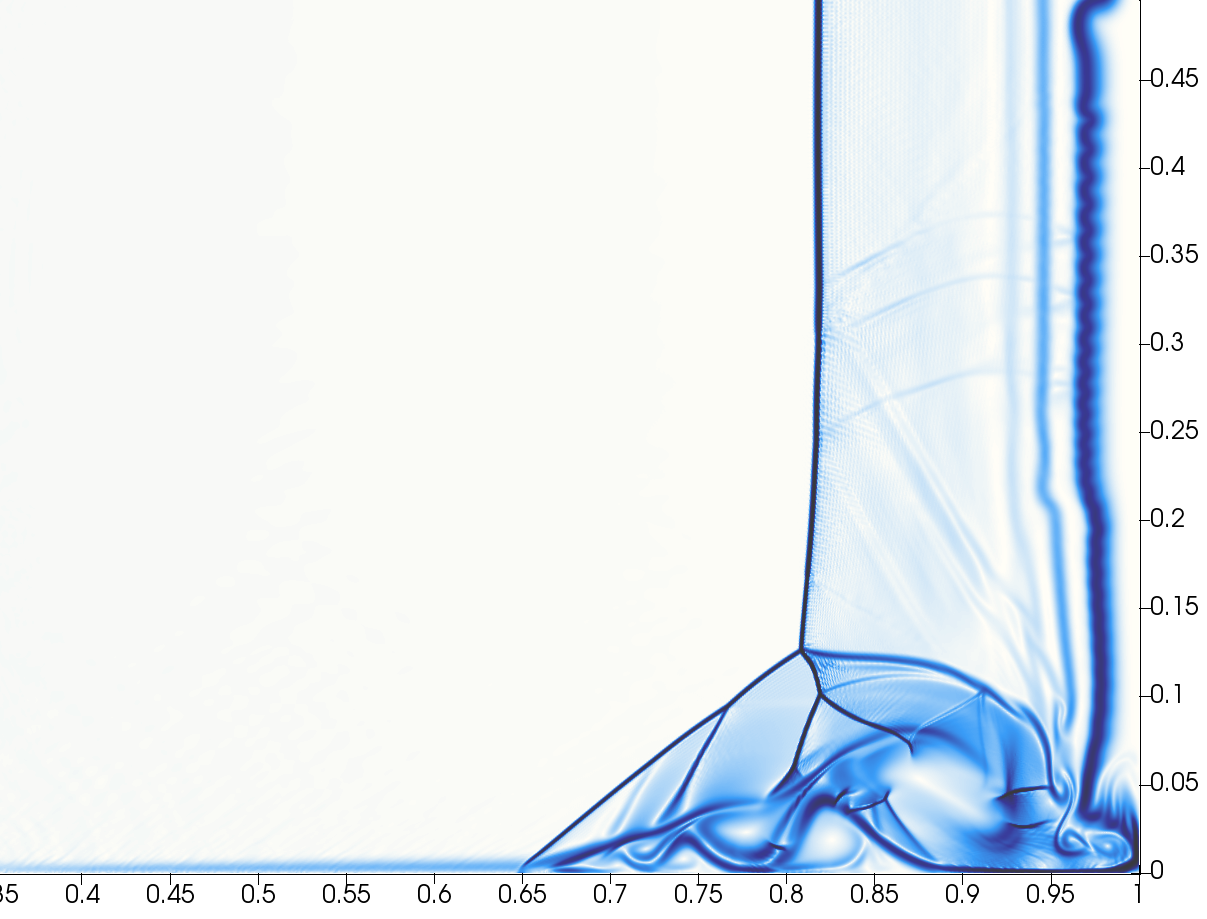}}}
      \hspace{0.2em}
      \subfloat[Mesh 1, $t=0.8$.]
      {\fbox{\includegraphics[width=0.32\textwidth,trim = 320 0 0 535,clip]{\FIGS/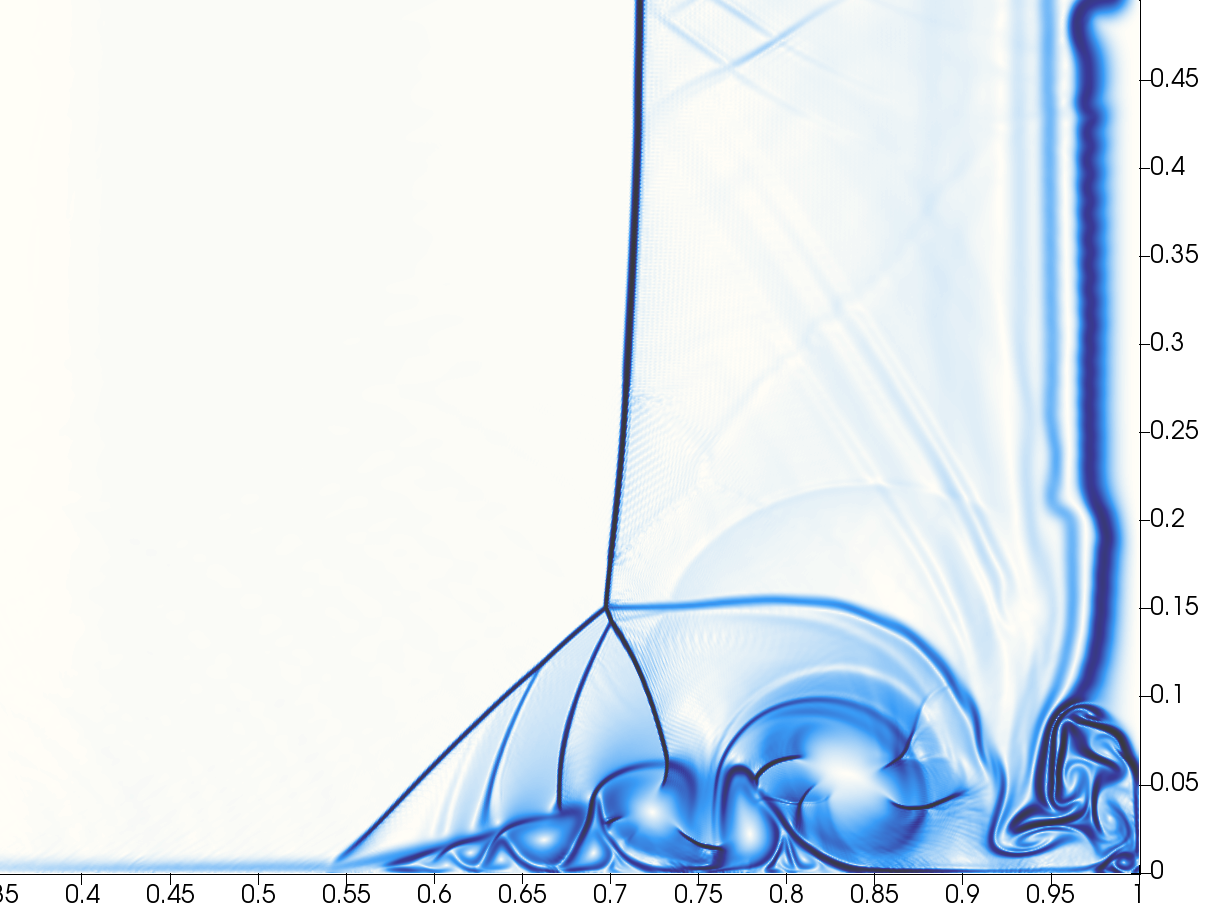}}}
      \hspace{0.2em}
   \subfloat[Mesh 1, $t=1$.]
      {\fbox{\includegraphics[width=0.32\textwidth,trim = 155 0 20 400,clip]{\FIGS/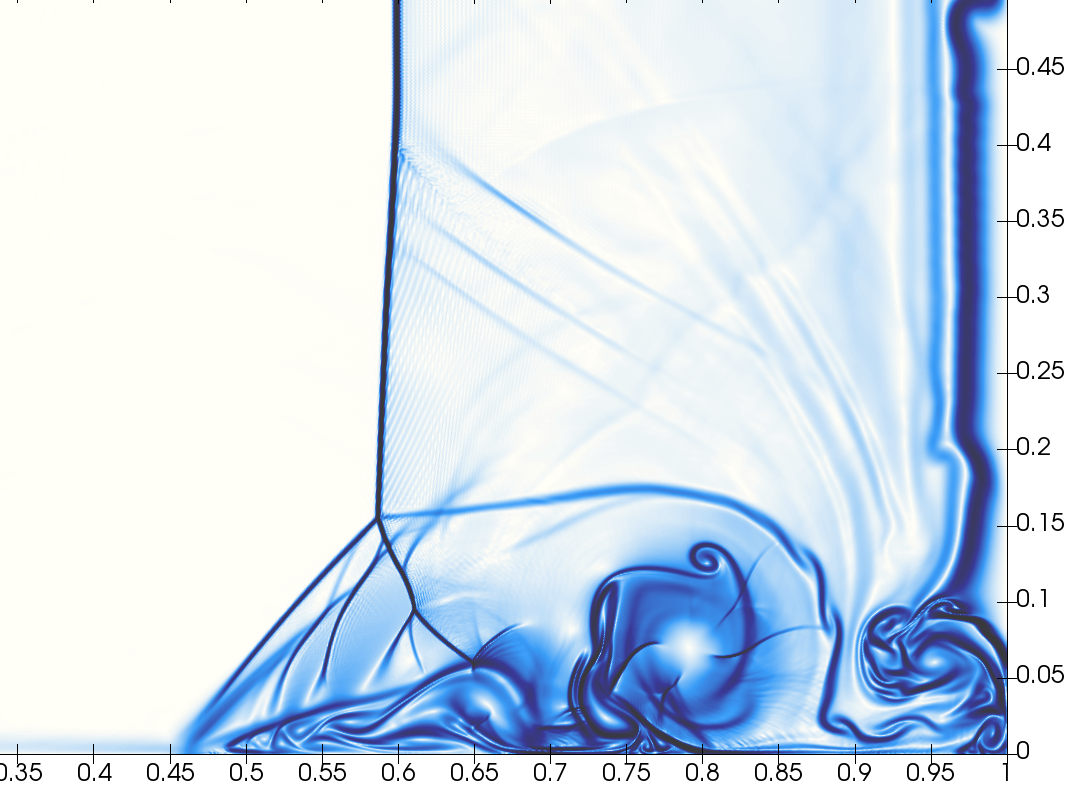}}}\\[-5pt]
    \subfloat[Mesh 2, $t=0.6$.]
      {\fbox{\includegraphics[width=0.32\textwidth,trim = 450 0 0 590,clip]{\FIGS/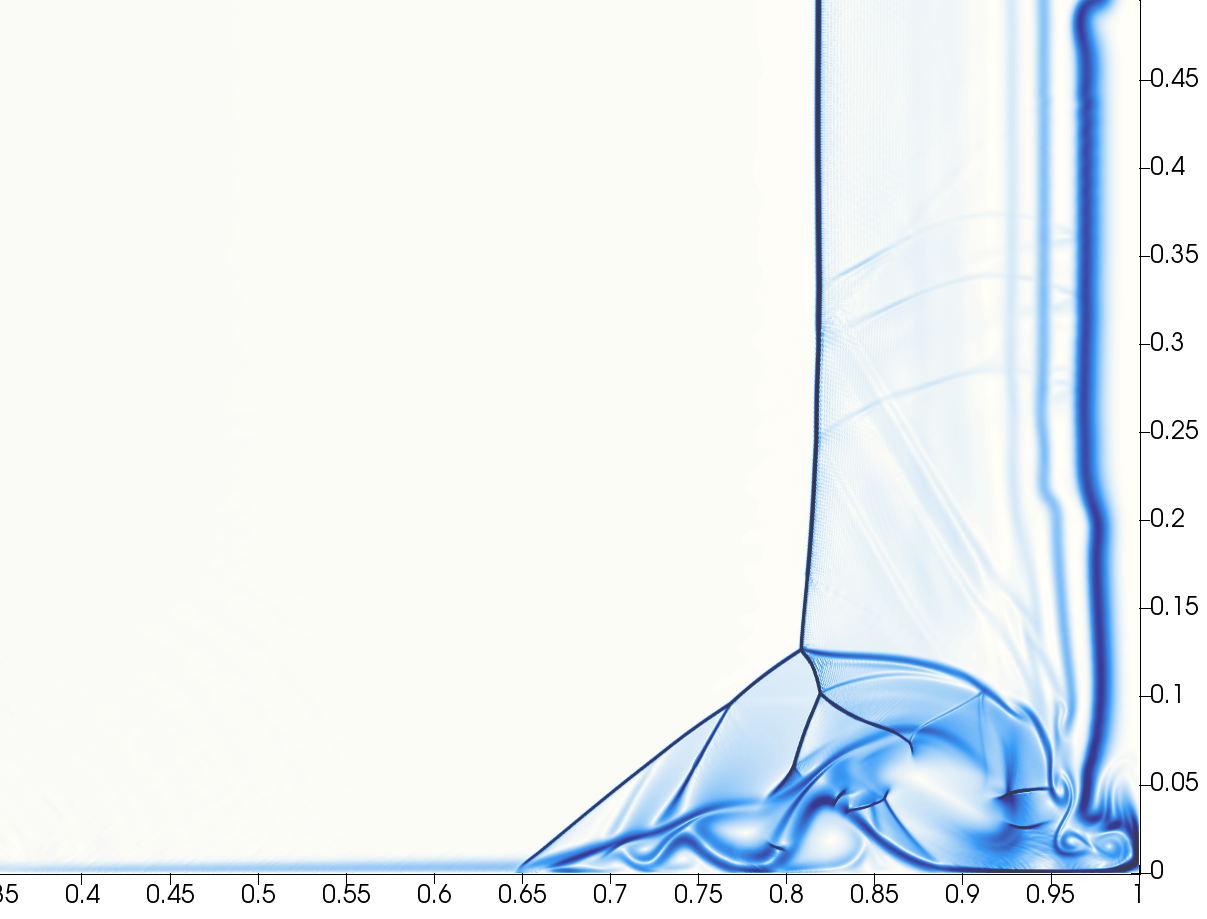}}}
      \hspace{0.2em}
      \subfloat[Mesh 2, $t=0.8$.]
      {\fbox{\includegraphics[width=0.32\textwidth,trim = 320 0 0 535,clip]{\FIGS/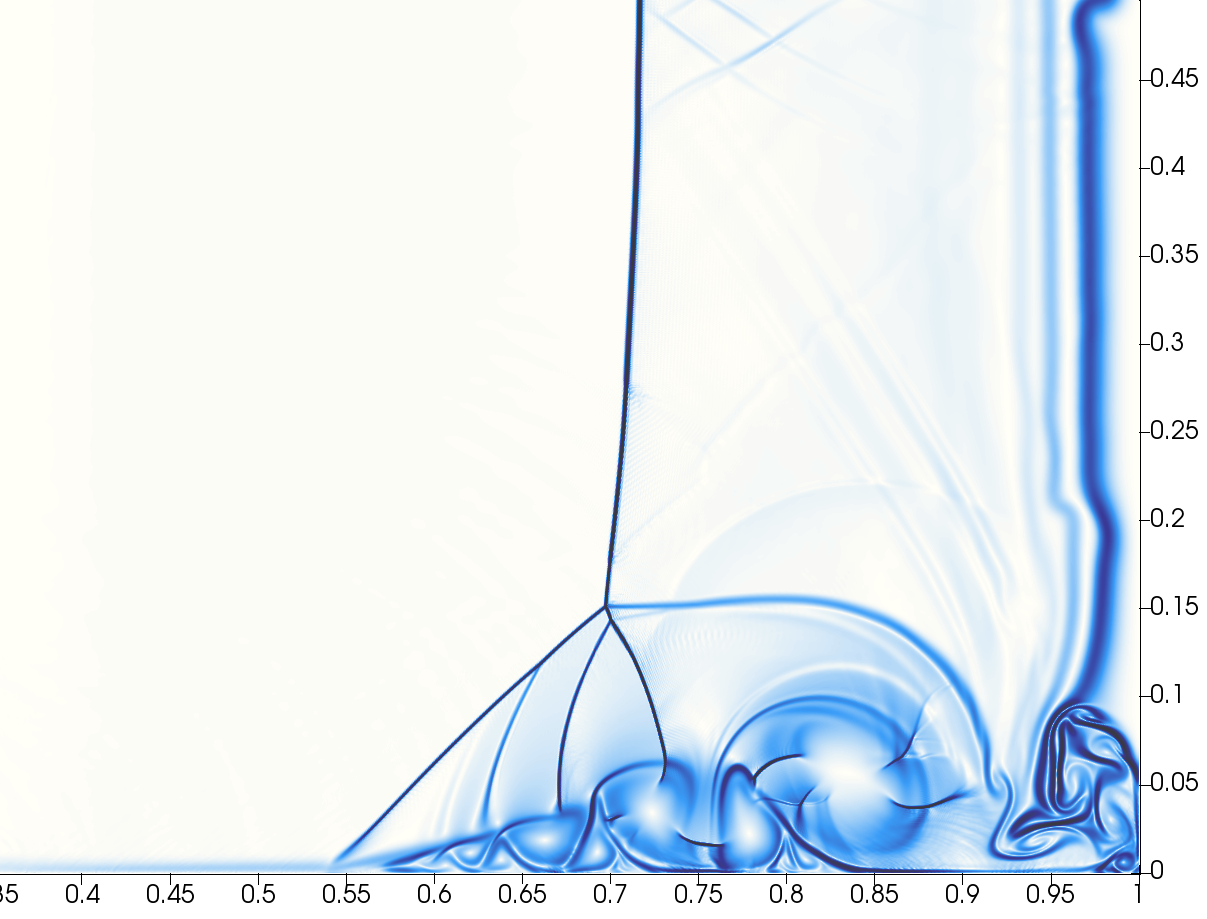}}}
      \hspace{0.2em}
      \subfloat[Mesh 2, $t=1$.]
      {\fbox{\includegraphics[width=0.32\textwidth,trim = 120 0 0 449,clip]{\FIGS/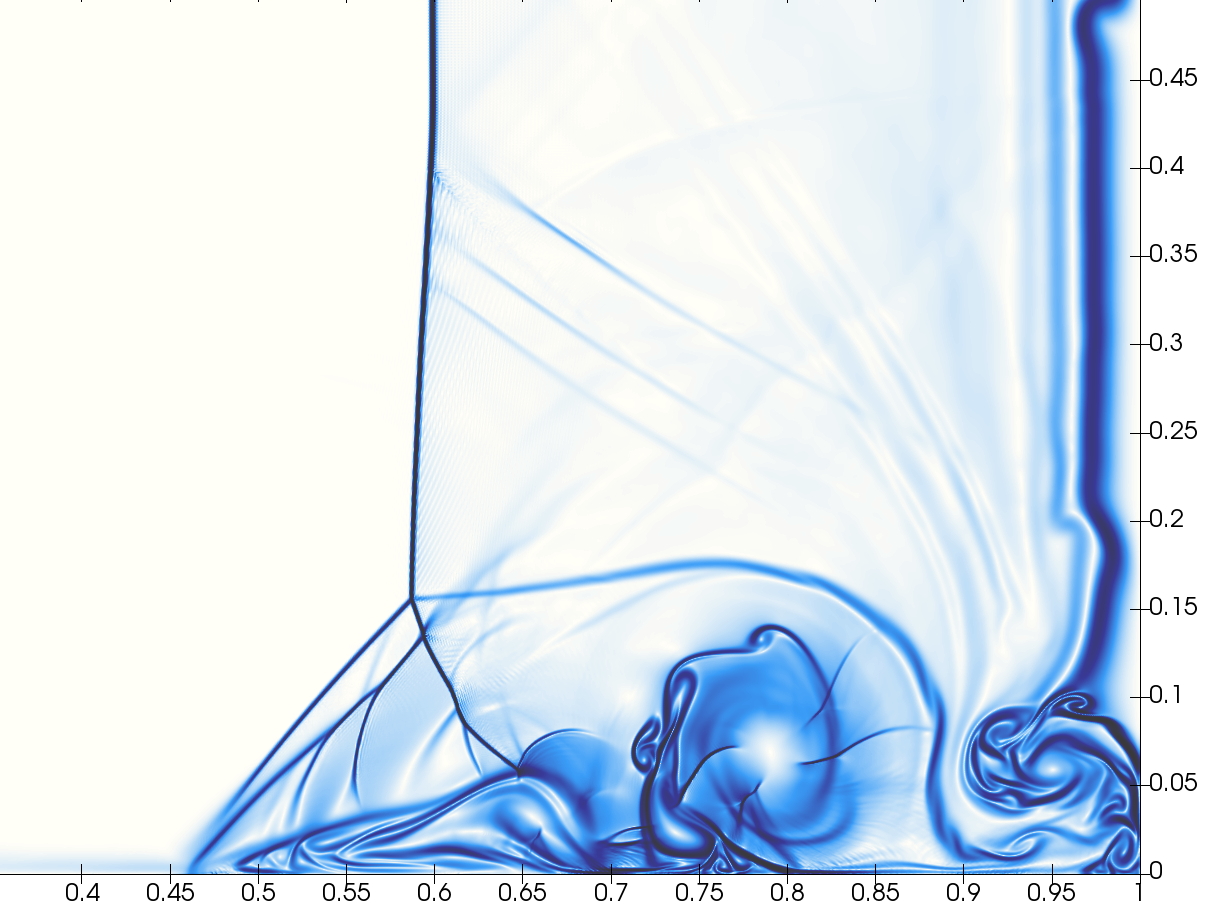}}}\\[-5pt]
      \subfloat[Mesh 3, $t=0.6$.]
      {\fbox{\includegraphics[width=0.32\textwidth,trim = 450 0 0 590,clip]{\FIGS/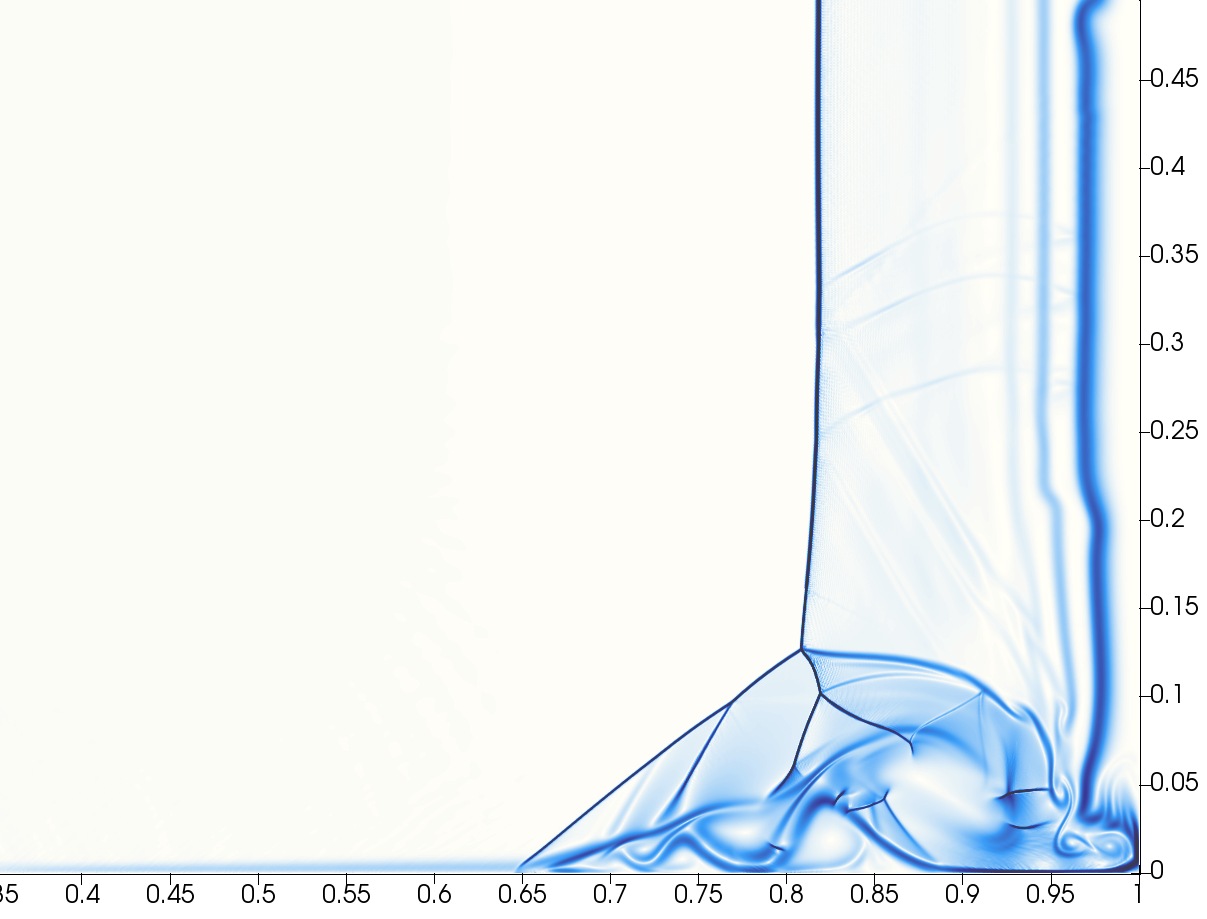}}}
      \hspace{0.2em}
      \subfloat[Mesh 3, $t=0.8$.]
      {\fbox{\includegraphics[width=0.32\textwidth,trim = 320 0 0 535,clip]{\FIGS/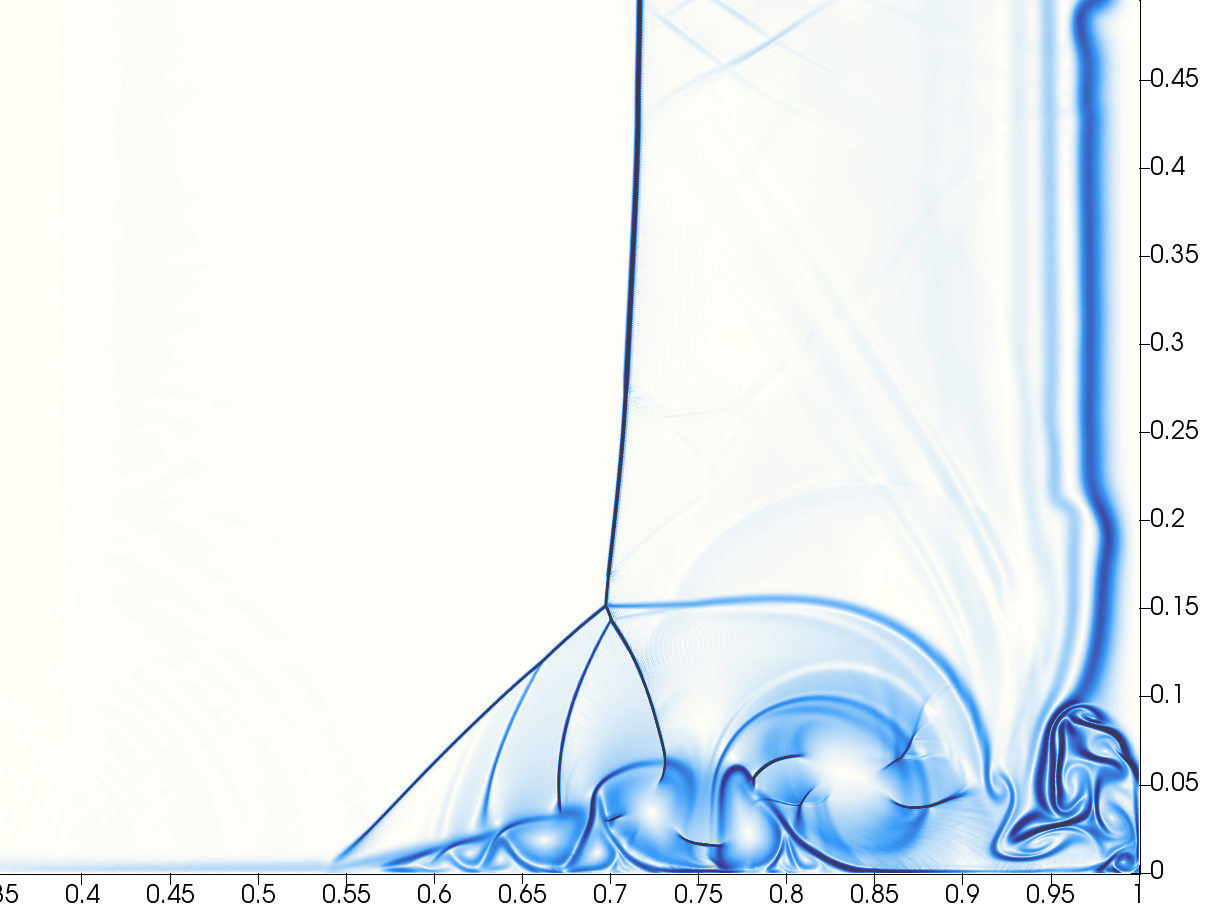}}}
      \hspace{0.2em}
    \subfloat[Mesh 3, $t=1$.]
      {\fbox{\includegraphics[width=0.32\textwidth,trim = 120 0 0 449,clip]{\FIGS/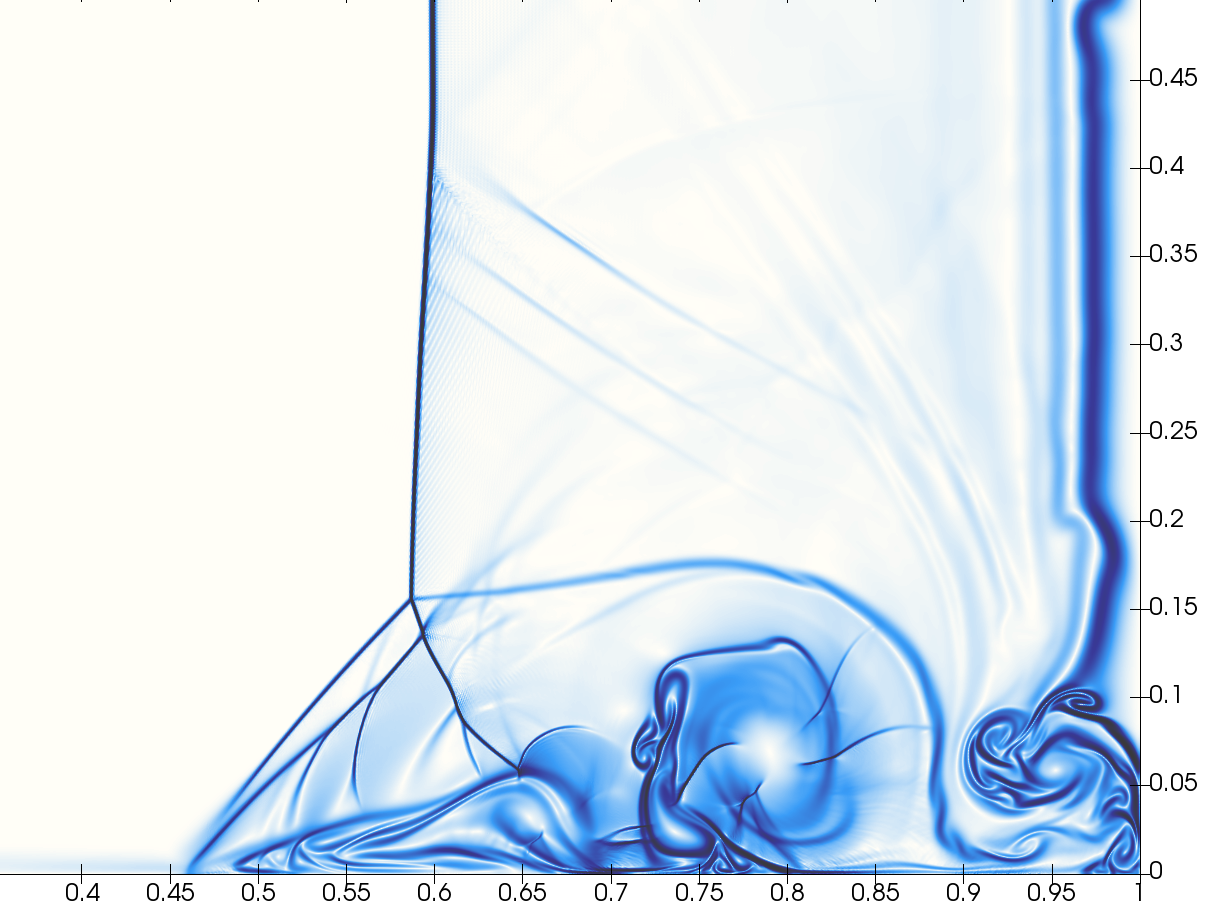}}\label{Fig:2D_SW_RE1000_iso_mesh3_t1}}
  \end{center}
  \caption{2D shocktube test. Density at $t\in\{0.6,0.8,1\}$ with $\mu=10^{-3}$. Meshes with
    increasing refinement level: Mesh~1, $359388$ grid point; Mesh~2, $684996$ grid point; Mesh~3, $859765$ grid points.}
  \label{Fig:2D_SW_RE1000}%
\end{figure}
We start by demonstrating the behavior of the method under nonuniform
mesh refinement. We show in Figure~\ref{Fig:2D_SW_RE1000} the gradient
of the density field at $t\in\{0.6,0.8,1\}$ for the three meshes:
Mesh~1 to Mesh~3. More precisely, denoting
$g(\bx)=\|\GRAD \rho_h(\bx)\|_{\ell^2}$,
$g_{\min}= \min_{\bx \in \Dom} g(\bx)$,
$g_{\max}= \max_{\bx \in \Dom} g(\bx)$, we visualize the quantity
$e^{-10\frac{g-g_{\min}}{g_{\max}-g_{\min}}}$ to amplify the
contrast. We observe that the results at $t=0.6$ and at $t=0.8$ vary
very little as the grids are refined. Some local changes are
  noticeable for the solution at $t=1$, but the overall structure of
  the flow seems to be converging when the meshsize decreases. There
is no real consensus yet in the literature on the solution at $t=1$
for $\mu=10^{-3}$.  For instance various schemes are tested in
\citet{Sjogreen_Yee_2003} on meshes ranging from $1000\CROSS 500$ grid
points to $4000\CROSS 2000$ grid points (in the half domain), but the
results reported therein seem to depend on the scheme that is chosen.
It is remarkable though that our results on the finest grid
  (Fig.~\ref{Fig:2D_SW_RE1000_iso_mesh3_t1}) are strikingly similar
to those reported Fig.~8d in \citet{Daru_Tenaud_2009}  and Fig.~11l in
\citet{Guangzhao_Kun_Feg_Phys_Fluids_2018} (see also Fig.~5a in \citep{Daru_Tenaud_2009} and Fig.~6c in
\citep{Guangzhao_Kun_Feg_Phys_Fluids_2018}); these three figures are
almost Xerox copies of each other. But none of the results reported in
\citep{Sjogreen_Yee_2003}  (and \citep{Kotov_etal_2014}) agree with the results shown in
Figure~\ref{Fig:2D_SW_RE1000} (and Fig.~8d in \citep{Daru_Tenaud_2009}
and Fig.~11l in \citep{Guangzhao_Kun_Feg_Phys_Fluids_2018}). In
  conclusion, it seems that our results agree very well with those reported
  in \cite{Daru_Tenaud_2009}
  and \cite{Guangzhao_Kun_Feg_Phys_Fluids_2018} but disagree with those
  reported in \citet{Sjogreen_Yee_2003} (and \citet{Kotov_etal_2014}), thereby shedding some doubts
  on the correctness of the computations in \citep{Sjogreen_Yee_2003,Kotov_etal_2014}.  If we are to believe
that in absence of vacuum the compressible Navier-Stokes equations in
two dimensions exhibit continuous dependence with respect to the
initial data, and it should therefore be possible to compute a
reference solution at $t=1$, then further computations with finer
meshes have to be done to clarify unambiguously and definitively the above issue.

\begin{figure}[hbt]\setlength{\fboxsep}{0pt}%
  \setlength{\fboxrule}{0.5pt}%
  \begin{center}
    \subfloat[$\mu=10^{-3}$]
      {\fbox{\includegraphics[width=0.49\textwidth,trim = 5 0 0 440,clip=]{\FIGS/RE1000_0p001_0004_t1.png}}}
    \hspace{0.3em}
    \subfloat[$\mu=5\CROSS 10^{-4}$]
      {\fbox{\includegraphics[width=0.49\textwidth,trim = 5 0 0 440,clip=]{\FIGS/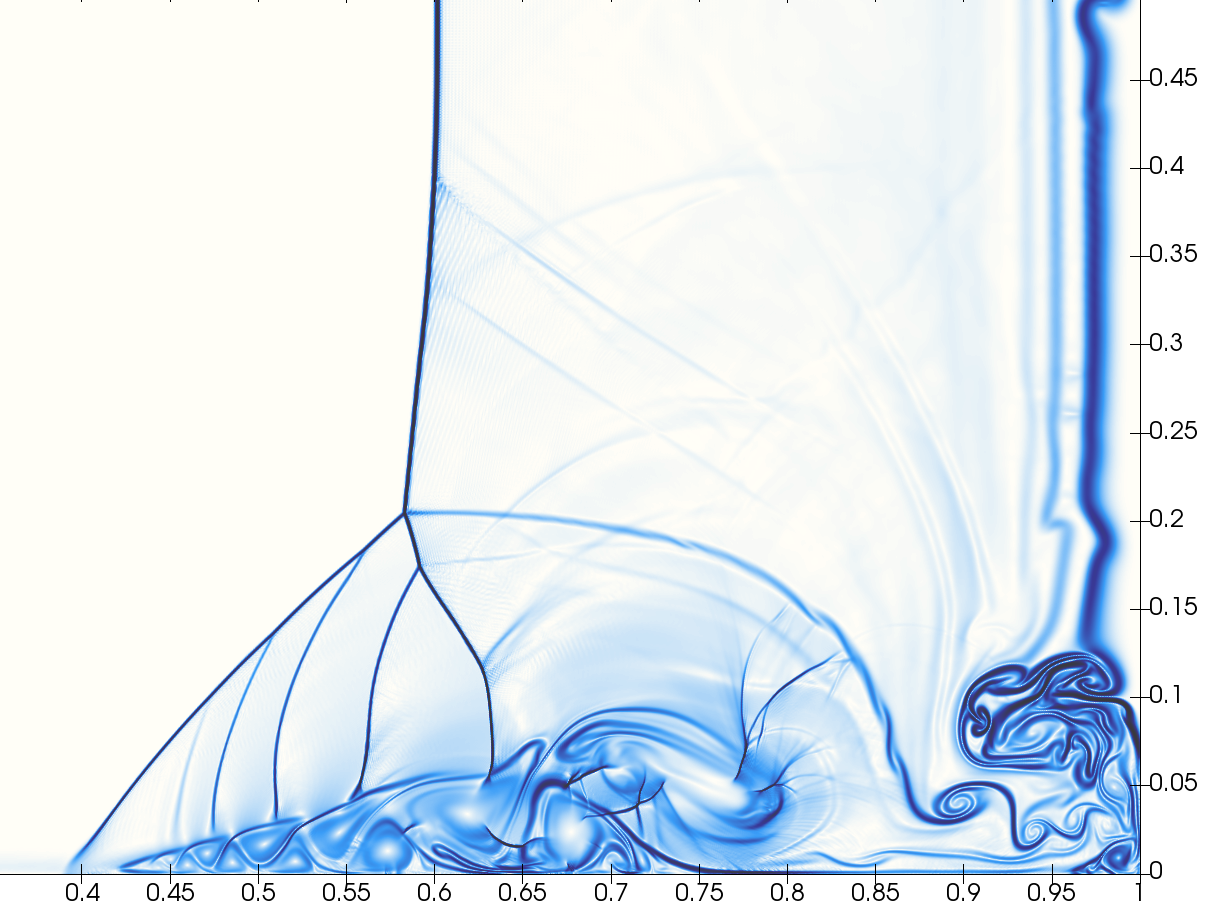}}} \\[-5pt]
    \subfloat[$\mu=2\CROSS 10^{-4}$]
      {\fbox{\includegraphics[width=0.49\textwidth,trim = 5 0 0 440,clip=]{\FIGS/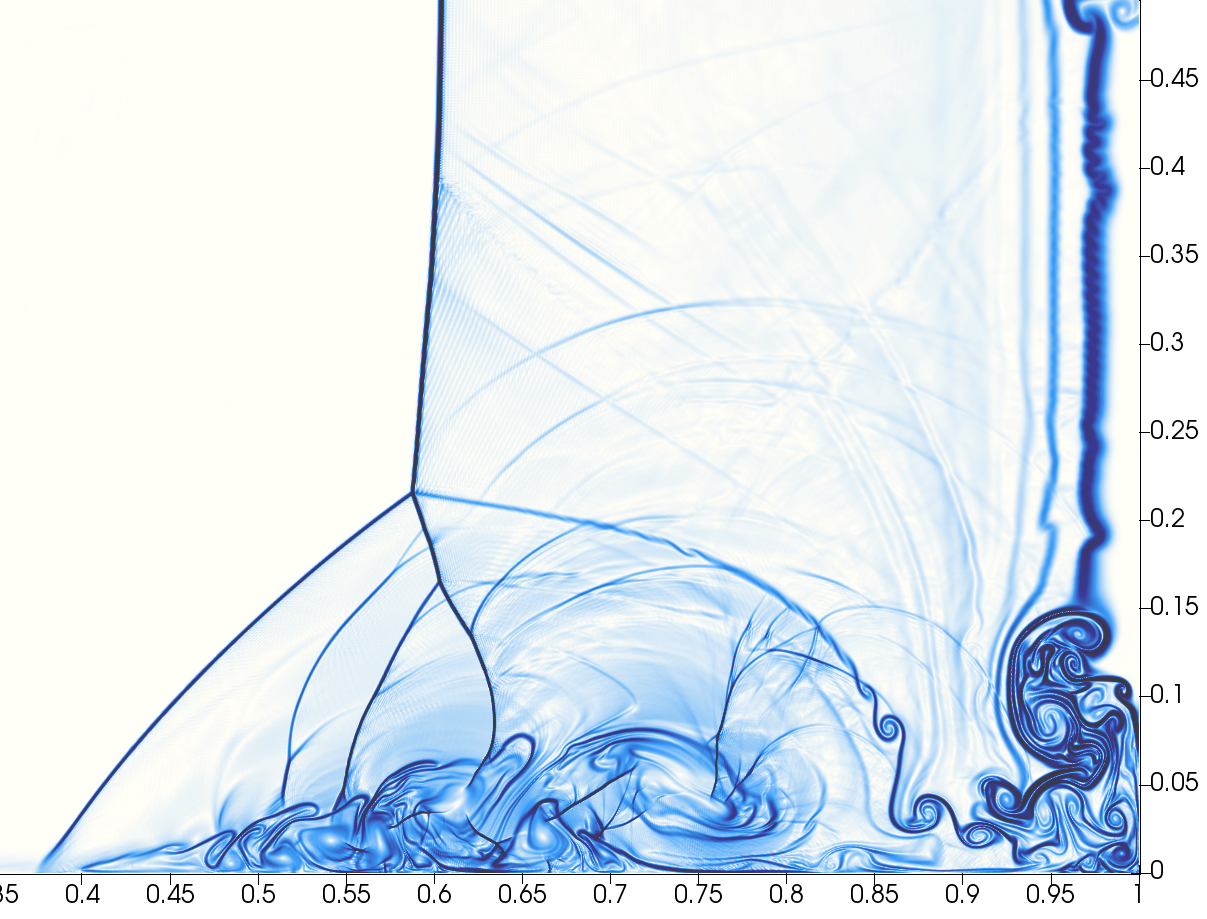}}}
    \hspace{0.3em}
    \subfloat[$\mu=10^{-4}$]
      {\fbox{\includegraphics[width=0.49\textwidth,trim = 5 0 0 440,clip=]{\FIGS/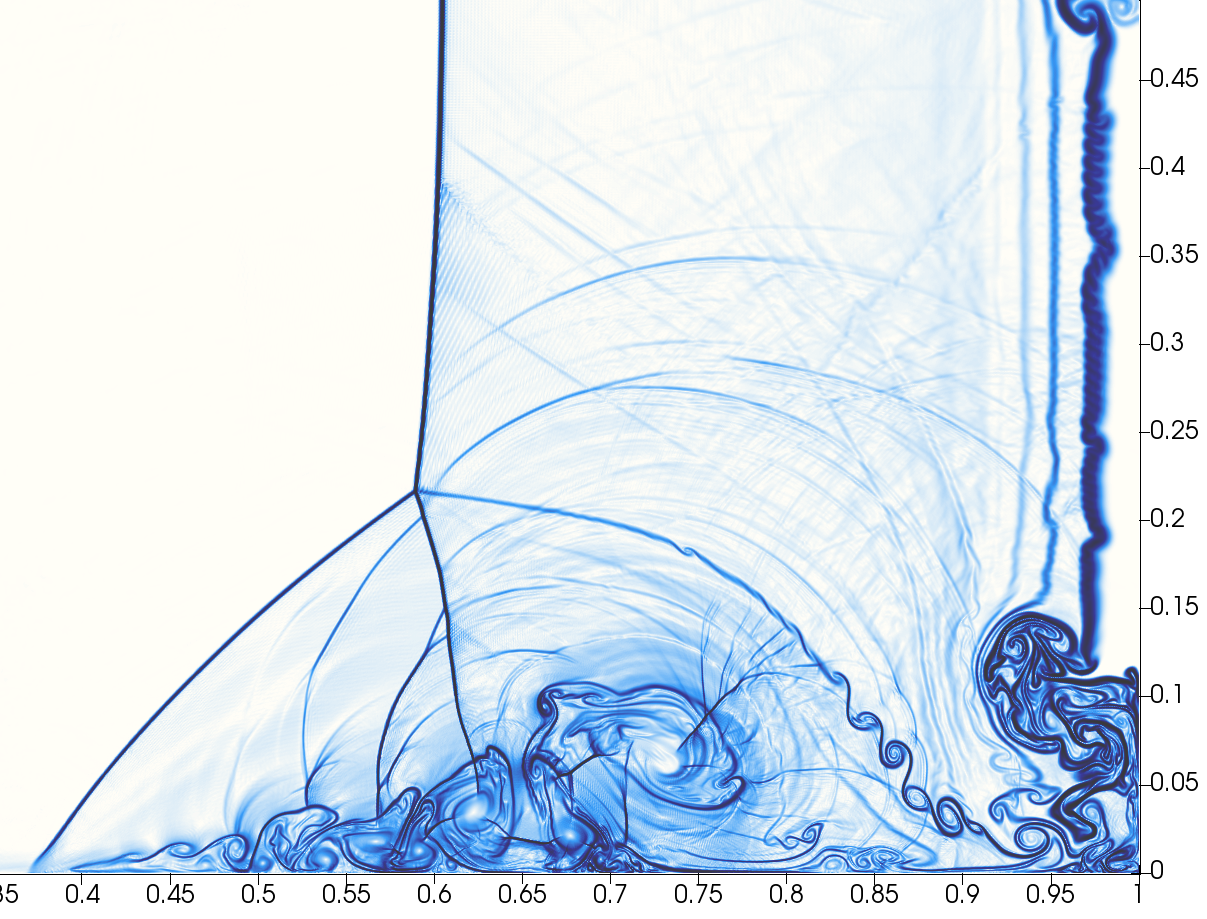}}}
  \end{center}
  \caption{%
    2D shocktube test, Mesh~3. Density at $t=1$ for $\mu\in \{10^{-3},5\CROSS 10^{-4},2\CROSS 10^{-4},10^{-4}\}$.}
  \label{Fig:2D_SW_RE1000_to_10000}
\end{figure}
As a last numerical illustration we recompute the density field at $t=1$ on
Mesh~4 for four increasingly smaller viscosities $\mu\in\{10^{-3},5\CROSS
10^{-4},2\CROSS 10^{-4},10^{-4}\}$. Results are reported in
Fig.~\ref{Fig:2D_SW_RE1000_to_10000}. We observe that for decreasing
viscosity the flow field develops increasingly more pronounced and smaller
vortex structures. This confirms that the influence of the artificial graph
viscosity of the hyperbolic step (see \S\ref{Sec:1D_viscous_shock}) is well
below viscous effects introduced by the physical viscosity $\mu$.


\section{Conclusions and Outlook}
\label{Sec:conclusions}
A fully discrete second-order order accurate method for solving the
compressible Navier-Stokes equations has been introduced. The novelty of
this work lies in the guaranteed invariant domain preservation of the fully
discrete method under the usual hyperbolic CFL condition. The method relies on
the operator-splitting strategy in order to preserve invariant set stability
properties. There is, in principle, no limitation for the accuracy in space. We
also notice that the method exhibits quite robust behaviour (in the eye-ball
norm) for flows containing strong shock interactions with viscous layers. At this
point in time, it is not yet clear how to develop a third-order accurate
(in-time) invariant-domain-preserving scheme.


\bibliographystyle{abbrvnat}

\end{document}